\documentclass[12pt,reqno]{amsart}
\usepackage{graphicx}

\usepackage{enumerate}
\usepackage{amsmath, amsthm} 
\usepackage{amsfonts}
\usepackage{amssymb}
\usepackage{fullpage} 
\usepackage{xcolor}
\usepackage{tikz}
\usepackage{stmaryrd}
\usetikzlibrary{cd}
\usetikzlibrary{decorations.pathmorphing}

\newtheorem{theorem}{Theorem}[section]

\newtheorem*{theorem*}{Theorem}

\newtheorem{lemma}[theorem]{Lemma}
\newtheorem{proposition}[theorem]{Proposition}
\newtheorem{corollary}[theorem]{Corollary}

\theoremstyle{definition}
\newtheorem{definition}{Definition}[section]

\newtheorem{remark}[definition]{Remark}

\usepackage{mathabx,amsmath,amscd, latexsym, amssymb, bbold}

\usepackage{amsmath, amssymb, tikz}

\def\bk{\mathbb{k}}
\def\bZ{\mathbb{Z}}
\def\cZ{\mathcal{Z}}
\def\cC{\mathcal{C}}
\def\cD{\mathcal{D}}

\def\spa{\mathrm{span}}

\DeclareMathOperator{\opp}{op}
\DeclareMathOperator{\id}{id}
\DeclareMathOperator{\op}{op}
\DeclareMathOperator{\Hom}{Hom}
\DeclareMathOperator{\Aut}{Aut}
\DeclareMathOperator{\Rep}{Rep}

\DeclareMathOperator{\coker}{coker}
\DeclareMathOperator{\ob}{ob}
\renewcommand{\Vec}{\mathrm{Vec}}

\DeclareMathOperator{\End}{End}

\DeclareMathOperator{\ncl}{ncl}

\def\incl{\hookrightarrow}

\def\rib{\tilde{\mathbf{v}}}
\def\deq{:=}
\def\one{\mathbb{1}}
\DeclareMathOperator{\Ker}{Ker}

\DeclareMathOperator{\im}{im}
\DeclareMathOperator{\Inv}{Inv}

\renewcommand{\epsilon}{\varepsilon}
\newcommand\trel{{\boxtimes}}

\DeclareMathOperator{\Res}{Res}

\DeclareMathOperator{\To}{\Rightarrow}

\newcommand{\mcE}{\mathcal{E}}

\newcommand{\mcM}{\mathcal{M}}

\title{Exact sequences of representation categories of weak Hopf algebras}
\date{}

\author{Quinn T. Kolt}

\address{Department of Mathematics, University of California, Santa Barbara. Santa Barbara, CA 93106, USA}
\email{quinn@math.ucsb.edu}

\begin{document}
    \begin{abstract}
        We study exact sequences of representation categories of weak Hopf algebras over an arbitrary field. Given a sequence $A\overset{k}{\to} B\overset{\pi}{\to} H$, where $A$ and $B$ are weak Hopf algebras and $H$ is a Hopf algebra, we develop verifiable algebraic conditions on $k$ and $\pi$ under which there is an exact sequence $\Rep(H)\to\Rep(B)\to\Rep(A)$ of tensor categories in the sense of Brugui\`eres and Natale (2011). Along the way, we develop a generalization of the restriction of scalars functor for maps $\pi:B\to H$ between associative algebras satisfying a weakened multiplicativity constraint depending on a relatively separable subalgebra $B_r\subseteq B$, as well as a theory of kernels and cokernels for weak Hopf algebras. We, in particular, find that the cokernel of a weak Hopf algebra homomorphism  $k:A\to B$ always exists, is a Hopf algebra, and the cokernel map is surjective when $A$ is connected. We conclude by studying examples of such exact sequences built from groupoids, formal ribbon extensions of quasitriangular weak Hopf algebras, and cocycled crossed products of a Hopf algebra acting weakly on a weak Hopf algebra. 
    \end{abstract}
    \maketitle
    \tableofcontents

    \section*{Introduction}
    In \cite{BruguieresNatale11}, the notion of an exact sequence of tensor categories was introduced. As an example, it was proven that a strictly exact sequence 
    \begin{equation}\label{wha-exact}
        \begin{tikzcd}
        A\arrow[r,"k",hook]&B\arrow[r,"\pi", two heads]&H
        \end{tikzcd}
    \end{equation}
    of finite-dimensional Hopf algebras gives rise to an exact sequence of finite tensor categories 
    \begin{equation}\label{rep-exact}
        \begin{tikzcd}
        \Rep(H)\arrow[r,"\pi^*",hook]&\Rep(B)\arrow[r,"k^*"]&\Rep(A),
        \end{tikzcd}
    \end{equation}
    where $\pi^*$ and $k^*$ denote the restriction of scalars of $\pi$ and $k$ respectively. The present article has one primary goal: to obtain a set of verifiable algebraic conditions for a sequence of the form \eqref{wha-exact} to produce an exact sequence of tensor categories of the form \eqref{rep-exact} when $A$ and $B$ are weak Hopf algebras and $H$ is a Hopf algebra. The first category in the sequence necessarily has a fiber functor, so it is necessarily representable by a Hopf algebra. This justifies our choice to assume $H$ is a Hopf algebra. This assumption is further justified by our proof that dominance and normality of $k^*$ are sufficient to prove exactness of \eqref{rep-exact} when $\pi:B\to H$ is the cokernel of $k$.
    
    One obstacle in this task is that kernels/cokernels don't exist in the category of weak Hopf algebras with weak Hopf algebra homomorphisms, even if we fix the minimal weak Hopf subalgebra. There is no general zero map, so a kernel/cokernel internal to this category is impossible. One resolution to construct a generalized kernel/cokernel is to not require that $\pi$ be an algebra homomorphism. Because of this, we cannot define the restriction of scalars in the classical sense. This necessitates a generalization of the restriction of scalars functor. Thus, the present article develops a theory of the ``relative'' restriction of scalars functor for a relatively multiplicative map $\pi$, kernels and cokernels for weak Hopf algebras, and conditions under which a sequence of the form \eqref{wha-exact} gives rise to an exact sequence of tensor categories of the form \eqref{rep-exact}, where $\pi:B\to H$ is the cokernel of the map $k:A\to B$.

    We make as few assumptions as possible on the weak Hopf algebras $A$ and $B$. Generally, we must assume that $A$ has a simple trivial representation so that $\Rep(A)$ is a tensor category. We work over an arbitrary field. We do not assume that $A$ nor $B$ is semisimple. We never assume finite-dimensionality of $A$ or $B$, just finite-generation of $B$ as an $A$-module (which does imply that $H$ is finite-dimensional, however). This theory is therefore built with semisimple and non-semisimple, zero characteristic and positive characteristic, and finite- and infinite-dimensional applications in mind.

    It is not the goal of this article to provide a set of conditions for defining a strictly exact sequence of weak Hopf algebras, i.e., to generalize the definition for Hopf algebras discussed in \cite{Schneider93}. The conditions we arrive at are more general in some ways and more strict in others. Right flatness of $B$ over $A$ is unnecessary, and our notion of normality is more general. In the other direction, we assume that $B$ is pure over $A$ and finitely generated as an $A$-module, which is not necessary for general strictly exact sequences of Hopf algebras. The dual of a strictly exact sequence of Hopf algebras is again a strictly exact sequence of Hopf algebras. One would hope that a strictly exact sequence of weak Hopf algebras satisfies the same property. However, like the exact sequences of tensor categories of \cite{BruguieresNatale11}, our sequences are not closed under duals.
    
    In this article, we use the original definition of an exact sequence of tensor categories, introduced in \cite{BruguieresNatale11}. There are more general notions of exact sequences of tensor categories, such as the one introduced in \cite{ETINGOF20171187}. In the theory we construct, there is a lot of room for generalization, so it is likely that a similar theory may be built beyond the original definition, and perhaps this perspective will motivate another definition of exact sequences of tensor categories or further provide evidence that one definition should be preferred. It would be interesting to determine when the relative restriction of scalars functor is (lax) monoidal and which functors between fusion categories arise this way. Even in this less general context however, there are many questions left unanswered. For example, which exact sequences of fusion categories arise through this construction? 

    While many find weak Hopf algebras cumbersome and prefer to study their representation categories abstractly \cite{EGNO}, weak Hopf algebras still find recent applications in physics \cite{Jia2026,Jia2025} and algebra \cite{Calderón2024,liu2025weak}. Given a fusion category $\cC$ and a $\cC$-module category $\mcM$, the strip algebra $\mathrm{Str}_{\mcM}(\cC)$ has a representation theory which captures the behavior of solitons in a system whose symmetry is captured by $\cC$ and boundary behavior is captured by $\mcM$ \cite{Córdova2025}. The strip algebras are weak Hopf algebras which satisfy $\Rep(\mathrm{Str}_{\mcM}(\cC))\cong \cC_{\mcM}^*$. In particular, $\Rep(\mathrm{Str}_{\cC}(\cC))\cong \cC$. If $\cC$ and $\mcM$ are taken to be unitary, then $\mathrm{Str}_{\mcM}(\cC)$ is a $C^*$-weak Hopf algebra. The tube algebra $\mathrm{Tube}(\cC)$ is a related weak Hopf algebra which realizes the Drinfeld center $\Rep(\mathrm{Tube}(\cC))\cong \cZ(\cC)$, though this is not a monoidal equivalence in general \cite{Jia2024}. While this article does not explore this direction, a natural question is if we may learn about the structure of $\mathrm{Str}_{\mcM}(\cC), \mathrm{Tube}(\cC), \mcM, \cC^*_{\mcM},$ or $\cC$ by studying exact sequences constructed by our procedure. For example, one may ask, given an exact sequence of fusion categories $\cC\to\cD\to\mcE,$ when there is a sequence of the form
    $$\mathrm{Str}_{\mcE}(\mcE)\to \mathrm{Str}_{\cD}(\cD)\to \End(\cC\to\mcE),$$
    which induces the given exact sequence by our construction. The last term $\End(\cC\to\mcE)$ in this sequence is the Hopf algebra of natural endo-transformations of the fiber functor $\cC\to\mcE$, which, by Tannaka-Krein reconstruction satisfies $\Rep(\End(\cC\to\mcE)) = \cC$.

    In Section \ref{sec:prelim}, we review background for exact sequences of tensor categories and weak Hopf algebras. In Section \ref{sec:relative-mult}, we introduce the relative restriction of scalars. This functor is defined for a map $\pi:A\to C$ between associative algebras, such that $A$ is unital and contains a relatively separable subalgebra $A_r$, and $\pi$ satisfies a relative multiplicativity constraint depending on $A_r$. In Section \ref{sec:ker-coker}, we discuss kernels and cokernels for weak Hopf algebras using these relatively multiplicative maps. Note that a weak Hopf algebra $A$ contains two separable subalgebras $A_t$ and $A_s$, so we assume that our maps are relatively multiplicative with respect to both $A_t$ and $A_s$. When $C=H$ is a Hopf algebra, we further assume that $\pi:A\to H$ is a coalgebra homomorphism and preserves the antipode. In this situation, the relative restriction of scalars is a tensor functor. In Section \ref{sec:exact-seq}, we construct an exact sequence of tensor categories given a sufficiently well-behaved weak Hopf algebra homomorphism. We provide an algebraic condition which is necessary and sufficient for normality of the functor $k^*$. In addition, we have another easy to check algebraic condition which implies normality. These conditions also imply that $\pi^*$ induces an equivalence between $\Rep(H)$ and $\Ker k^*$. Finally, we show a well-known algebraic condition implies dominance of $k^*$. In Section \ref{sec:examples}, we show that normal groupoid homomorphisms, the formal ribbon extension of a quasitriangular weak Hopf algebra, and cocycled crossed products all form exact sequences of representation categories, where the first and third assume some finiteness constraint.

    \section{Preliminaries}\label{sec:prelim}
    \subsection{Exact sequences of tensor categories}
    Fix a field $\bk$. By a tensor category $\cC$, we mean a locally finite, rigid, $\bk$-linear, abelian category with simple unit $\one$. By a fusion category $\cC$, we mean a finite semisimple tensor category. Given a ($\bk$-)algebra $A$, denote by $\Rep(A)$, the category of finite-dimensional $A$-modules. Denote by $A$-Mod, the category of all $A$-modules.  
    
    By a tensor functor between tensor categories $\cC$ and $\cD$, we mean an exact, faithful, linear, monoidal functor $F:\cC\to\cD$. For a tensor functor $F:\cC\to\cD$ between tensor categories, we say that $F$ is 
    \begin{itemize}
        \item \textit{normal} if for any object $c\in\cC$, there is a subobject $c_0$ such that $F(c_0)$ is the maximal trivial (isomorphic to $\one^{\oplus n}\in\cD$ for some $n$) subobject of $F(c)$;
        \item \textit{dominant} if for any object $d\in\cD$, there is an object $c\in \cC$ such that $d$ is a subobject of $F(c)$.
    \end{itemize}
    We define the kernel of $F:\cC\to\cD$ as the full subcategory of $\cC$ with the following objects:
    \begin{align*}
        \Ker F = \{c\in\cC | F(c)\cong \one^{\oplus n} \text{ for some }n\geq 0\}\subseteq \cC.
    \end{align*}
    Then, an exact sequence of tensor categories is a sequence of tensor functors
    $$\cC\overset{G}{\to}\cD\overset{F}{\to}\mcE$$
    such that
    \begin{enumerate}[(\alph*)]
        \item $G$ induces an equivalence $\cC\overset{\sim}{\to}\Ker F$,
        \item $F$ is normal and dominant.
    \end{enumerate}
    
    \subsection{Weak Hopf algebra basics}
    The essential theory behind weak Hopf algebras was established in \cite{Bhm1998WeakHA,BOHM2000156}. We review some of the basic definitions and properties of weak Hopf algebras. 
    
    Throughout this article, for a coalgebra $A$, we make use of sum-less Sweedler's notation, where we write $\Delta(a) = a^{(1)}\otimes a^{(2)}$. This does not denote a simple tensor but rather a sum of simple tensors.
    \begin{definition}\label{wha-def}
        Let $A$ be a unital associative algebra as well as a counital coassociative coalgebra over a field $\bk$. Then, $A$ is a \textit{weak bialgebra} if these structures have the following compatibility:
        \begin{enumerate}
            \item $\Delta(a_1)\Delta(a_2) = \Delta(a_1a_2)$,
            \item $\Delta^2(1) = (\Delta(1)\otimes 1)(1\otimes\Delta(1))=(1\otimes\Delta(1))(\Delta(1)\otimes 1)$,
            \item $\epsilon(a_1a_2a_3) = \epsilon(a_1a_2^{(1)})\epsilon(a_2^{(2)}a_3) = \epsilon(a_1a_2^{(2)})\epsilon(a_2^{(1)}a_3)$.
        \end{enumerate}
        Furthermore, $A$ is a \textit{weak Hopf algebra} if it is a weak bialgebra with a linear map $S:A\to A$ such that
        \begin{enumerate}
            \setcounter{enumi}{3}
            \item $a^{(1)}S(a^{(2)}) = \epsilon(1^{(1)}a)1^{(2)}$,
            \item $S(a^{(1)})a^{(2)} = 1^{(1)}\epsilon(a1^{(2)})$,
            \item $S(a^{(1)})a^{(2)}S(a^{(3)}) = S(a)$.
        \end{enumerate}
    \end{definition}

    \begin{proposition}[\cite{Bhm1998WeakHA}]
        The following conditions are equivalent for a weak Hopf algebra $A$:
        \begin{enumerate}
            \item $A$ is a Hopf algebra,
            \item $\epsilon(a_1a_2) = \epsilon(a_1)\epsilon(a_2)$,
            \item $\Delta(1) = 1\otimes 1$,
            \item $S(a^{(1)})a^{(2)} = \epsilon(a)1$,
            \item $a^{(1)}S(a^{(2)}) = \epsilon(a)1$.
        \end{enumerate}
    \end{proposition}
    
    For a weak Hopf algebra $A$, the \textit{target} $\epsilon_t:A\to A$ and \textit{source} $\epsilon_s:A\to A$ maps are defined, for $a\in A$, as 
    \begin{align*}
        \epsilon_t(a) &:= a^{(1)}S(a^{(2)}) = \epsilon(1^{(1)}a)1^{(2)} = S(1^{(1)})\epsilon(1^{(2)}a),\\
        \epsilon_s(a) &:= S(a^{(1)})a^{(2)} = 1^{(1)}\epsilon(a1^{(2)}) = \epsilon(a1^{(1)})S(1^{(2)}).
    \end{align*}
    Define the target $A_t = \im \epsilon_t$ and source $A_s = \im \epsilon_s$ algebras so that $\epsilon_t:A\to A_t$ and $\epsilon_s:A\to A_s$. These algebras are separable subalgebras of $A$ with the following separability idempotents:
    $$e_t = S(1^{(1)})\otimes 1^{(2)},\hspace{50pt}e_s = 1^{(1)}\otimes S(1^{(2)}).$$
    $A_t$ is the smallest subspace of $A$ such that $\Delta(1)\in A\otimes A_t$. Similarly, $A_s$ is the smallest subspace such that $\Delta(1)\in A_s\otimes A.$ In particular, $\Delta(1)\in A_s\otimes A_t$. Like in the Hopf algebra case, the antipode is multiplication and comultiplication reversing, and preserves the unit and counit. The antipode also defines a unital algebra isomorphism $S|_{A_s}:A_s\to A_t^{\opp}$ and $S|_{A_t}:A_t\to A_s^{\opp}$. 

    The source and target maps satisfy a handful of useful properties, which apply for any $a, a_1, a_2\in A$;
    \begin{enumerate}
        \item $\epsilon_t$ and $\epsilon_s$ are unit and counit preserving.
        \item $\epsilon_t$ and $\epsilon_s$ commute under multiplication, i.e., $\epsilon_t(a_1)\epsilon_s(a_2) = \epsilon_s(a_1)\epsilon_t(a_2).$
        In particular, $A_s$ and $A_t$ commute.
        \item Products combine as follows:
        \begin{enumerate}
            \item $\epsilon_t(a_1)\epsilon_t(a_2) = \epsilon_t(\epsilon_t(a_1)a_2)$,
            \item $\epsilon_s(a_1)\epsilon_s(a_2) = \epsilon_s(a_1\epsilon_s(a_2))$.
        \end{enumerate}
        \item Products inside $\epsilon$ have the following relation:
        \begin{enumerate}
            \item $\epsilon(a_1a_2) = \epsilon(a_1\epsilon_t(a_2))$,
            \item $\epsilon(a_1a_2) = \epsilon(\epsilon_s(a_1)a_2)$.
        \end{enumerate}
        \item Similarly, products inside $\epsilon_t$ or $\epsilon_s$ have the following relation:
        \begin{enumerate}
            \item $\epsilon_t(a_1a_2) = \epsilon_s(a_1\epsilon_t(a_2))$,
            \item $\epsilon_s(a_1a_2) = \epsilon_t(\epsilon_s(a_1)a_2)$.
        \end{enumerate}
        \item Comultiplications expand as follows:
        \begin{enumerate}
            \item $\Delta(\epsilon_t(a)) = 1^{(1)}\epsilon_t(a)\otimes 1^{(2)}$,
            \item $\Delta(\epsilon_s(a)) = 1^{(1)}\otimes \epsilon_s(a)1^{(2)}$.
        \end{enumerate}
        \item Source/target maps on one factor of a comultiplication may be expressed as:
        \begin{enumerate}
            \item $a^{(1)}\otimes\epsilon_t(a^{(2)}) = 1^{(1)}a\otimes 1^{(2)} = a^{(1)}\otimes a^{(2)}S(a^{(3)})$,
            \item $\epsilon_s(a^{(1)})\otimes a^{(2)} = 1^{(1)}\otimes a1^{(2)} = S(a^{(1)})a^{(2)}\otimes a^{(3)}$.
        \end{enumerate}
        \item $\epsilon_t$ and $\epsilon_s$ are idempotent under composition, i.e., $\epsilon_t\circ\epsilon_t = \epsilon_t$ and $\epsilon_s\circ\epsilon_s=\epsilon_s$.
        \item We have the following relation between compositions of $\epsilon_t$, $\epsilon_s$, and $S$:
        \begin{enumerate}
            \item $\epsilon_t\circ S=S\circ\epsilon_s = \epsilon_t\circ\epsilon_s$,
            \item $\epsilon_s\circ S=S\circ\epsilon_t = \epsilon_s\circ\epsilon_t$.
        \end{enumerate}
        \item $\epsilon_t$ and $\epsilon_s$ can be commuted across comultiplication up to an application of the antipode:
        \begin{enumerate}
            \item $a_1^{(1)}\epsilon_s(a_2)\otimes a_1^{(2)} = a_1^{(1)}\otimes a_1^{(2)}S(\epsilon_s(a_2))$,
            \item $S(\epsilon_t(a_2))a_1^{(1)}\otimes a_1^{(2)} = a_1^{(1)}\otimes \epsilon_t(a_2)a_1^{(2)}$.
        \end{enumerate}
    \end{enumerate}    
    The algebra $A_{\min} := A_tA_s$ is actually a weak Hopf subalgebra of $A$. In particular, this means $\Delta(A_t) \subseteq A_{\min}\otimes A_t$ and $\Delta(A_s)\subseteq A_s\otimes A_{\min}$. Moreover, $A_{\min}$ is minimal in that it contains no weak Hopf subalgebras other than itself. All other weak Hopf subalgebras of $A$ contain $A_{\min}$, so it is the unique minimal weak Hopf subalgebra of $A$. Such weak Hopf algebras necessarily have an invertible antipode. In particular, $S|_{A_{\min}}:A_{\min}\to A_{\min}$ is invertible. The classification of minimal weak Hopf algebras over an algebraically closed field of characteristic 0 is discussed in \cite{NIKSHYCH2002257}. 

    For a weak Hopf algebra $A$, denote $A^+_t=\ker\epsilon_t^A$ and $A^+_s=\ker\epsilon_s^A$. Both $A_t^+$ and $A_s^+$ are coideals of $A$. Moreover, $A_t^+\subset A$ is a left ideal, and $A_s^+\subset A$ is a right ideal. In particular, $A_t^+A_s^+$ is as associative algebra.

    A \textit{weak Hopf algebra homomorphism} $k:A\to B$ is a unital algebra homomorphism $k:A\to B$ which is also a counital coalgebra homomorphism and preserves the antipode.
    \begin{lemma}\label{targets-iso}
        If $k:A\to B$ is a weak Hopf algebra homomorphism, then $k$ induces unital algebra isomorphisms $k|_{A_t}:A_t\overset{\sim}{\to} B_t$ and $k|_{A_s}:A_s\overset{\sim}{\to} B_s$ and a weak Hopf algebra isomorphism $k|_{A_{\min}}:A_{\min}\overset{\sim}{\to} B_{\min}$.
    \end{lemma}

    \subsection{The representation theory of weak Hopf algebras}
    The representation category of a semisimple weak Hopf algebra is a multifusion category. Every multifusion category arises this way \cite[Cor. 2.22]{eno2005fusion}. More generally, given a finitely semisimple module category $\mcM$ over a multifusion category $\cC$, there is a finite-dimensional semisimple weak Hopf algebra $A$ such that $\Rep(A)\cong\cC$ as tensor categories and $\Rep_{\Rep(A)}(A_t)\cong\mcM$ as $\Rep(A)$-module categories. The case where the representation category is fusion or more generally tensor is of particular interest in this article, as exact sequences are not defined in the more general context. 

    Let $A$ be a weak Hopf algebra. Given two $A$-modules $M, N$, there is an $A$-module structure on $1^{(1)}M\otimes 1^{(2)}N$ generated by
    $$a\cdot (1^{(1)}m\otimes 1^{(2)}n) = a^{(1)}m\otimes a^{(2)}n.$$
    $A_s$ and $A_t$ have isomorphic $A$-module structures. Given $a\in A$, $x\in A_t$ and $y\in A_s$, the following are well-defined actions of $A$ on $A_t$ and $A_s$:
    $$a\cdot_{A_t} x := \epsilon_t(ax), \hspace{20pt} a\cdot_{A_s} y := \epsilon_s(yS(a)).$$
    The $A$-module $A_t\cong A_s$ acts as the tensor unit in $\Rep(A)$ under the tensor product described earlier. The category $\Rep(A)$ is also rigid and linear, but, because $A_t$ is not necessarily simple as an $A$-module, $\Rep(A)$ need not be a tensor category. In many of our results, however, simplicity of $A_t$ as an $A$-module is assumed, aligning with our desired application to exact sequences of tensor categories.
    
    \begin{proposition}[\cite{EGNO}]\label{connected}
        Let $A$ be a weak Hopf algebra. Then, the following are equivalent:
        \begin{enumerate}[(\alph*)]
            \item $A$ has an irreducible trivial representation,
            \item $Z(A)\cap A_t = \bk 1_A$, and
            \item $\End_A(A_t) = \bk\id_{A_t}$.
        \end{enumerate}
    \end{proposition}
    If $A$ satisfies the equivalent conditions of Proposition \ref{connected}, we say that $A$ is \textit{connected}. Now, we define the set of left invariants of an $A$-module $M$, which will allow us to identify trivial objects inside $\Rep(A)$.
    
    \begin{definition}
        $\Inv_A(M) = \{m\in M \mid a\cdot m = \epsilon_t(a)\cdot m, \forall a\in A\}$.
    \end{definition}
    We can equivalently write 
    $$\Inv_A(M) = \{m\in M \mid A_t^+\cdot m = 0\} = \{f(1_A) | f\in\Hom_A (A_t\to M)\}.$$ 
    Note that $\Inv_A(M)$ is a right $A$-module but not a left $A$-module in general, unlike the Hopf algebra case, because $\epsilon_t$ does not play as well with multiplication. However, the $A$-module $A\cdot \Inv_A(M)$ has particular significance in this article for the following reason.
    \begin{lemma}\label{Hinv-trivial}
        Let $A$ be a weak Hopf algebra and $M\in\Rep(A)$. The maximal submodule of $M$ isomorphic to a quotient of $A_t^{\oplus n}$ for some $n$ is $A\cdot \Inv_A(M)=A_t\cdot \Inv_A(M).$
    \end{lemma}
    When $A$ is connected, Lemma \ref{Hinv-trivial} says that $A\cdot \Inv_A(M)=A_t\cdot \Inv_A(M)$ is the maximal trivial submodule of $M$.
    \begin{proof}
        Observe that
        $$A\cdot \Inv_A(M) = \epsilon_t(A)\cdot\Inv_A(M) = A_t\cdot\Inv_A(M).$$
        Consider $A_t\otimes \Inv_A(M)$ as a left $A$-module with the action $a\cdot (x\otimes \Lambda) = \epsilon_t(ax)\otimes\Lambda$. Clearly,
        $$A_t^{\oplus\dim\Inv_A(M)} \cong A_t\otimes \Inv_A(M)$$
        as $A$-modules. Moreover, there is an obvious map $A_t\otimes \Inv_A(M)\to A_t\cdot\Inv_A(M)$ defined by $x\otimes \Lambda\mapsto x\cdot\Lambda$, which can be readily verified to be $A$-linear and surjective. Consequently, $A_t\cdot\Inv_A(M)$ is a quotient of $A_t^{\oplus\dim\Inv(A)}$.
    
        Finally, we show that $A_t\cdot \Inv_A(M)$ contains any such submodule $N\subseteq M$. Suppose $A_t^{\oplus n}/Q\cong N$, and let $f_i$ be the following composition:
        $$f_i:A_t\incl A_t^{\oplus n}\twoheadrightarrow A_t^{\oplus n}/Q\cong N\incl M$$
        where the inclusion sends elements of $A_t$ into the $i$th coordinate of $A_t^{\oplus n}$. Then, since the $f_i:A_t\to M$ are $A$-linear, we have
        $$N=\sum f_i(A_t) = \sum A_t\cdot f_i(1_A) \subseteq A_t\cdot \Inv_A(M).$$
    \end{proof}

    \section{Relatively multiplicative maps and restriction of scalars} \label{sec:relative-mult}   \subsection{Relatively multiplicative maps and basic properties}
    \begin{definition}
        Let $A_r\subseteq A$ be an inclusion of unital associative algebras. We say that this inclusion is \textit{relatively separable} if there is $e=e^{(1)}\otimes e^{(2)}\in A_r\otimes A$ such that 
        \begin{enumerate}[(\alph*)]
            \item $xe^{(1)}\otimes e^{(2)} = e^{(1)}\otimes e^{(2)}x$ for any $x\in A_r$,
            \item $e^{(1)}e^{(2)} = 1_A$.
        \end{enumerate}
    \end{definition}
    Technically, one should call this \textit{left} relatively separable. For this paper, right relative separability would not be sufficient. We will see that the property that the first tensor component of $e$ is contained in $A_r$ is necessary for defining the relative restriction of scalars functor $\Rep(B) \to \Rep(A)$. Right relative separability would work if we wanted a functor mod-$B\to$ mod-$A$. In the case of weak Hopf algebras, $A_t$ and $A_s$ end up being separable in their own right, so this relative generality ultimately isn't used here. However, we prove a handful of properties that hold in this level of generality.

    \begin{definition}
        A \textit{relatively multiplicative map} from a unital associative algebra $A$ to an associative algebra $C$ relative to a relatively separable subalgebra $A_r\subseteq A$ is a linear map $\pi:A\to C$ such that, for all $a, b\in A$,
        $$\pi(ab) = \pi(ae^{(1)})\pi(e^{(2)}b).$$
    \end{definition}
    
    \begin{proposition}\label{mult-implies-rel}
        Any multiplicative map $\pi:A\to C$ is a relatively multiplicative map relative to any relatively separable subalgebra of $A$. 
    \end{proposition}

    Relatively multiplicative maps are not closed under composition. In particular, there is no category of ``relatively separable subalgebras'' and ``relatively multiplicative maps.'' Later, we define birelatively multiplicative and Hopf maps, which again are not closed under composition. However, these maps do play well with algebra homomorphisms, as shown in the following proposition.

    \begin{proposition}\label{mult-comp-rel}\label{rel-comp-mult}\label{tensor-rel}
        One can build new relatively multiplicative maps as follows:
        \begin{enumerate}[(\alph*)]
            \item Let $\pi:A\to C$ be a relatively multiplicative map and $g:C\to C'$ be a multiplicative map. Then, $g\circ\pi$ is a relatively multiplicative map.
            \item Let $k:A\to A'$ be a multiplicative map such that $(k\otimes k)(e) = e'$ and $\pi:A'\to C$ be a relatively multiplicative map. Then, $\pi\circ k$ is a relatively multiplicative map.
            \item If $\pi:A\to C$ and $\pi':A'\to C'$ are relatively multiplicative maps, then $\pi\otimes \pi':A\otimes A'\to C\otimes C'$ is also relatively multiplicative.
        \end{enumerate}
    \end{proposition}
    
    \subsection{Relative restriction of scalars}
    By Proposition \ref{mult-implies-rel}, we can interpret a relatively multiplicative map as a generalization of a multiplicative map. Unital algebra homomorphisms $\pi:A\to C$ between algebras $A$ and $C$ induce a functor $\pi^*:\Rep(C)\to \Rep(A)$ called the restriction of scalars. In this section, given a relatively multiplicative map $\pi:A\to C$, we can still define a functor $\pi^*:\Rep(C)\to \Rep(A)$. In the case that $\pi$ is unital and multiplicative, this new functor agrees with the restriction of scalars functor (Proposition \ref{mult-res}), which motivates the name ``relative restriction of scalars.''

    \begin{theorem}
        Let $\pi:A\to C$ be a relatively multiplicative map and $M$ be a left $C$-module. Then, define 
        $$\pi^*(M) = A_r\otimes M / \spa\{x\otimes m - e^{(1)}\otimes \pi(e^{(2)}x)m \mid x\in A_r, m\in M\}.$$
        Then, 
        \begin{enumerate}[(\alph*)]
            \item $\pi^*(M)$ has a left $A$-module structure with action given by
        $$a\cdot (x\trel m) = e^{(1)}\trel \pi(e^{(2)}ax)m,$$
            where we denote the image of $x\otimes m$ under the quotient map $A_r\otimes M\to\pi^*(M)$ by $x\trel m$;
            \item for a $C$-linear map $f:M\to N$ between $C$-modules $M, N$, there is a unique $A$-linear map $\pi^*(f)$ which makes the following diagram commute:
        \begin{center}
            \begin{tikzcd}
                A_r\otimes M\arrow[r, "\id_{A_r}\otimes f"]\arrow[d, "\text{quot.}",two heads] & A_r\otimes N\arrow[d, "\text{quot.}",two heads]\\
                \pi^*(M)\arrow[r,"\pi^*(f)"]&\pi^*(N);
            \end{tikzcd}
        \end{center}
            \item $\pi^*$ is functorial.
        \end{enumerate}
    \end{theorem}
    \begin{proof}
        (a): By construction,
        $$1\cdot (x\trel m) = e^{(1)}\trel \pi(e^{(2)}x)m = x\trel m.$$
        Associativity follows from relative multiplicativity:
        $$a\cdot (b\cdot (x\trel m)) = e^{(1)}\trel \pi(e^{(2)}ae^{(1')})\pi(e^{(2')}bx)m = e^{(1)}\trel \pi(e^{(2)}abx)m = ab\cdot (x\trel m).$$
        Thus, $\pi^*(M)$ has an $A$-module structure.

        (b): Given $f:M\to N$, consider the composition $\tilde f:A_r\otimes M\to A_r\otimes N\to \pi^*(N).$ By $C$-linearity of $f$,
        $$\tilde f(x\otimes m - e^{(1)}\otimes \pi(e^{(2)}x)m) = x\trel f(m) - e^{(1)}\trel \pi(e^{(2)}x)f(m) = 0,$$
        so $\tilde f$ factors through a linear map $\pi^*(f):\pi^*(M)\to \pi^*(N)$ defined by $x\trel m\mapsto x\trel f(m)$. Once again, it follows by relative multiplicativity that this map is $A$-linear.

        (c): Obvious by construction of $\pi^*(f)$.
    \end{proof}
    
    Note that, whether or not $C$ is unital, the action of $A$ on $\pi^*(M)$ is unit-preserving. The choice of relatively multiplicative subalgebra does affect the functor $\pi^*$ in general, if, say, $\pi$ was relatively multiplicative with respect to two distinct relatively separable subalgebras. However, this is not the case when $\pi$ is multiplicative.
    \begin{proposition} \label{mult-res}
        Let $\pi:A\to C$ be a unital algebra homomorphism. Then, $\pi^*\cong \Res_\pi$, the usual restriction of scalars regardless of the choice of relatively separable subalgebra $A_r\subseteq A$.
    \end{proposition}
    \begin{proof}
        Observe that $$x\trel m = e^{(1)}\trel \pi(e^{(2)})\pi(x)m = 1\trel \pi(x)m.$$
        Thus, $\iota:M\incl\pi^*(M)$ is a surjective map. Conversely, the map $x\trel m\mapsto \pi(x)m$ preserves $x\trel m = e^{(1)}\trel \pi(e^{(2)}x)m$, so $\iota$ is also injective. One may readily verify that $\iota$ is also an $A$-linear and defines a natural isomorphism $\pi^*\cong \Res_\pi$. 
    \end{proof}
    Proposition \ref{mult-res} justifies our use of the notation of $\pi^*$ and calling $\pi^*$ the relative restriction of scalars. Moving forward, we use the notation $f^*$ for both restriction of scalars of a unital algebra homomorphism $f$ and relative restriction of scalars of a relatively multiplicative map $f$ because there is no ambiguity.
    
    \begin{proposition}
        Suppose $\pi:A\to C$ is a relatively multiplicative map and $g:C\to C'$ is a multiplicative map. Then, $(g\circ \pi)^{*} \cong \pi^{*}\circ g^*$.
    \end{proposition}
    \begin{proof}
        Unraveling definitions, one sees that, in fact, these functors are identical. That is, $(g\circ \pi)^{*} = \pi^*\circ g^{*}$.
    \end{proof}
    
    \begin{proposition}\label{rel-comp-mult-functor}
        Suppose $k:A\to A'$ is a unital algebra homomorphism such that $(k\otimes k)(e) = e'$, and $\pi:A'\to C$ is a relatively multiplicative map. Then, there is a natural transformation $\eta:(\pi\circ k)^{*}\To k^*\circ \pi^{*}$. Moreover, $\eta$ is a natural isomorphism if $k$ restricts to an isomorphism $k|_{A_r}:A_r\overset{\sim}{\to} A'_r$.
    \end{proposition}
    \begin{proof}
        Let $M\in\Rep(C)$. Then, an element in $(\pi\circ k)^*(M)$ may be expressed as a sum of elements of the form $x\trel m$, where $x\in A_r$ and $m\in M$ with $A$-action given by
        $$a\cdot_A (x\trel m) = e^{(1)}\trel (\pi\circ k)(e^{(2)}ax)m = e^{(1)}\trel \pi(k(e^{(2)}a)k(x))m.$$
        Meanwhile, an element in $(k^*\circ\pi^*)(M)$ may be expressed as a sum of elements of the form $y\trel m$, where $y\in A'_r$ and $m\in M$ with $A$-action given by
        $$a\cdot_A (y\trel m) = k(a)\cdot_{A'} (y\trel m) = (e')^{(1)}\trel \pi((e')^{(2)}k(a)y)m = k(e^{(1)})\trel \pi(k(e^{(2)}a)y)m.$$
        It is easy to see that $\eta_M:x\trel m\mapsto k(x)\trel m$ generates a well-defined $A$-linear map. Moreover, it is clearly natural. Finally, we see that this map is invertible if $k$ is invertible on $A_r$.
    \end{proof}
    
    \begin{proposition}\label{counit-case}
        Suppose $\epsilon:A\to \bk$ is a relatively multiplicative map, and define $\epsilon_r:A\to A_r$ to be, for $a\in A$, 
        $$\epsilon_r(a)\deq e^{(1)}\epsilon(e^{(2)}a).$$
        Then, $\epsilon_r$ defines a multiplicative action of $A$ on $A_r$ via $a\cdot x = \epsilon_r(ax)$. Moreover, $1 = \epsilon_r(1)$ if and only if this action on $A_r$ is unital. In this case, $\epsilon^*:\Vec\to\Rep(A)$ is equivalent to the functor generated by $\bk\mapsto A_r$.
    \end{proposition}
    
    \begin{corollary}\label{epsAB-case}
        Suppose $C$ is a unital associative algebra and $\epsilon:A\to \bk$ is a relatively multiplicative map. Let $\epsilon_{A, C}:A\to C$ is defined by $\epsilon_{A,C}(a)=\epsilon(a)1_C$. Then, $\epsilon_{A, C}$ is a relatively multiplicative map and
        $$\epsilon_{A,C}^*(M) \cong  A_r^{\oplus\dim(M)}.$$
    \end{corollary}
    \begin{proof}
        Note that $1_C:\bk\to C$ is multiplicative and $1_C^*$ is the forgetful functor $\Rep(C)\to\Vec$. By Proposition \ref{counit-case}, $\epsilon_A^*$ is the inclusion $\Vec\to\Rep(A)$, so
        $$\epsilon_{A,C}^*(M) \cong (\epsilon_A^*\circ 1_C^*)(M) \cong  A_r^{\oplus\dim(M)}.$$
    \end{proof}
    
    \section{Kernels and cokernels for weak Hopf algebras}\label{sec:ker-coker}
    
    \subsection{Hopf maps}
    \begin{definition}
        Let $A$ be a weak Hopf algebra and $C$ be an associative algebra. Then, a linear map $\pi:A\to C$ is \textit{birelatively multiplicative} if $\pi$ is a relatively multiplicative map relative to both $A_t\subseteq A$ and $A_s\subseteq A$ with their canonical separability idempotents. Explicitly, $\pi$ is birelatively multiplicative if, for all $a_1, a_2\in A$,
        $$\pi(a_1S(1^{(1)}))\pi(1^{(2)}a_2) = \pi(a_1a_2) = \pi(a_11^{(1)})\pi(S(1^{(2)})a_2).$$
        Suppose $H$ is a Hopf algebra. Then, a \textit{Hopf map} $\pi:A\to H$ is a birelatively multiplicative, antipode-preserving, coalgebra homomorphism.
    \end{definition}

    Note that $\epsilon_{A, H}:x\mapsto \epsilon_A(x) 1_H$ is Hopf for any weak Hopf algebra $A$ and Hopf algebra $H$. However, it is not comultiplicative or counital when $H$ is a weak Hopf algebra. However, in this case, the map is still birelatively multiplicative and antipode-preserving.
    
    \begin{proposition}\label{hopf-map-properties}
        Suppose $\pi:A\to H$ is a Hopf map. Then, for any $x_1, x_2\in A_t$ and $a, a_1, a_2, a_3\in A$, we have the following identities:
        \begin{enumerate}[(\alph*)]
            \item $\pi(\epsilon_t(a))=\epsilon_{A,H}(x)=\pi(\epsilon_s(a))$,
            \item $\pi(S(x_1)ax_2)=\pi(x_1aS^{-1}(x_2))$,
            \item $\pi(a_1\epsilon_t(a_2)a_3) = \pi(a_1a_2^{(2)})\pi(S(a_2^{(2)})a_3)$,
            \item $\pi(a_1\epsilon_s(a_2)a_3) = \pi(a_1S(a_2^{(1)}))\pi(a_2^{(2)}a_3)$.
        \end{enumerate}
    \end{proposition}
    When discussing the relative restriction of scalars $\pi^*:\Rep(C)\to\Rep(A)$ of a birelatively multiplicative map $\pi:A\to C$, we mean relative to $A_t$.

    \begin{lemma}\label{hopf-quotient-is-trivial}
        Suppose $\pi:A\to H$ is a Hopf map. Then, for an $H$-module $M$, $\pi^*(M)\cong A_t\otimes M$ as an $A_{t}$-module, with the action of left multiplication on the left tensor component.
    \end{lemma}
    \begin{proof}
        For $x\in A_t$ and $m\in M$,
        $$x\trel m - xS(1^{(1)})\trel \pi(1^{(2)})m = x\trel m - xS(1^{(1)})\trel \epsilon(1^{(2)})m = 0,$$
        so the quotient is the full space $A_t\otimes M$. Of course, for $y\in A_t$.
        $$y\cdot (x\trel m) = S(1^{(1)})\trel \pi(1^{(2)}yx)m = yx\trel m.$$
    \end{proof}
    Comparing to Proposition \ref{mult-res}, it seems Hopf maps are on the opposite end of relative multiplicativity (relative to $A_t$) from weak Hopf algebra homomorphisms. This aligns with the observation that weak Hopf algebra homomorphisms preserve the target subalgebra while Hopf maps collapse the target subalgebra to $\bk 1$.
        
    \begin{theorem}\label{tensor-functor}
        Suppose $\pi:A\to H$ is a Hopf map. Then, $\pi^{*}$ is a tensor functor with unitor $\pi^{*}(\bk_\epsilon)\to A_t$ given by $x\trel 1_{\bk_\epsilon}\to x$ and tensorator $\pi^{*}(M\otimes^H N)\to \pi^{*}(M)\otimes^A \pi^{*}(N)$ given by $x\trel (m\otimes n)\mapsto 1_A^{(1)}(x\trel m)\otimes 1_A^{(2)}(1\trel n)$.
    \end{theorem}
    \begin{proof}
        First, note that, for $M\in \Rep(H)$, there is a vector space isomorphism $\pi^{*}(M)\cong A_t\otimes M$ by Lemma \ref{hopf-quotient-is-trivial}. Moreover, an $H$-linear map $f:M\to N$ is sent to $\pi^*(f) = \id_{A_t}\otimes f$.  Therefore, $\pi^*$ is exact, faithful, and $\bk$-linear because $\otimes$ is exact, faithful, and $\bk$-linear in $\Vec$.
        
        It is easy to see that the given unitor is invertible and $A$-linear. Note that 
        $$1_A^{(1)}(x\trel m)\otimes 1_A^{(2)}(1\trel n) = (S(1_A^{(1')})\trel \pi(1_A^{(2')}x1_A^{(1)})m) \otimes (1_A^{(2)}\trel n) = (xS(1_A^{(1)})\trel m) \otimes (1_A^{(2)}\trel n).$$
        Thus, the inverse map of the given tensorator
        $$(xS(1_A^{(1)})\trel m) \otimes (1_A^{(2)}\trel n)\mapsto xS(1_A^{(1)})1_A^{(2)}\otimes (m \trel n) = x\otimes (m \trel n)$$
        is well-defined as a linear map. Next, we show that the given tensorator is $A$-linear:
        \begin{align*}
            a\cdot (x\trel (m\otimes n)) &= S(1_A^{(1)}) \trel (\pi\otimes\pi)(\Delta(1_A^{(2)}ax))(m\otimes n)\\
            &= S(1_A^{(1)}) \trel(\pi(1_A^{(2)}a^{(1)}x^{(1)})m\otimes \pi(1_A^{(3)}a^{(2)}x^{(2)})n)\\
            &\mapsto (S(1_A^{(1)})S(1_A^{(1')})\trel \pi(1_A^{(2)}a^{(1)}x^{(1)})m)\otimes (1_A^{(2')}\trel \pi(1_A^{(3)}a^{(2)}x^{(2)})n)\\
            &= (S(1_A^{(1')})\trel\pi(1_A^{(2)}1_A^{(1''')}a^{(1)}x1_A^{(1'')})m)\otimes (1_A^{(2')}S(1_A^{(1)})\trel \pi(1_A^{(2''')}a^{(2)}1_A^{(2'')})n)\\
            &= (S(1_A^{(1')})\trel\pi(1_A^{(2')}a^{(1)}xS(1_A^{(1'')}))m)\otimes (S(1_A^{(1)})\trel \pi(1_A^{(2)}a^{(2)}1_A^{(2'')})n)\\
            &= a\cdot ((xS(1_A^{(1'')})\trel m)\otimes (1_A^{(2'')}\trel n)).
        \end{align*}
        Verifying naturality and compatibility with the associators and unitors is straightforward. 
    \end{proof}

    \subsection{The cokernel of a weak Hopf algebra homomorphism}\label{coker-section}
    The cokernel map of a weak Hopf algebra homomorphism $k:A\to B$ is defined to be the birelatively multiplicative map $\pi:B\to H$ which is universal among birelatively multiplicative maps $B\to H'$ to unital associative algebras $H'$. We would expect the representation category of $H$ to have a fiber functor, at least in the case when we can achieve an exact sequence of tensor categories. To our delight, this fiber functor is easily identifiable. Remarkably, $H$ always has a Hopf algebra structure and is also universal with respect to Hopf maps $B\to H'$ to Hopf algebras $H'$.
    \begin{theorem}[Cokernels for weak Hopf algebra homomorphisms] \label{coker}
    Let $k:A\to B$ be a weak Hopf algebra homomorphism. Then, 
    \begin{enumerate}[(\alph*)]
        \item there is a ``cokernel'' unital associative algebra $H$ and birelatively multiplicative map $\pi:B\to H$ such that the following diagram commutes:
        \begin{equation}\label{coker-diagram}
            \begin{tikzcd}
                &&B\arrow[dd, "\pi"]\\\\
                A\arrow[rruu, "k"]\arrow[rr, "\epsilon_{A,H}"'] && H
            \end{tikzcd}
        \end{equation}
        \item $H$ is ``universal'' among such unital associative algebras, in the sense that, for any unital associative algebra $H'$ and birelatively multiplicative map $\pi':B\to H'$ for which Diagram \eqref{coker-diagram} commutes, there is a unique unital algebra homomorphism $\gamma:H\to H'$ such that the following diagram commutes:
        \begin{equation}\label{universal-diagram}
            \begin{tikzcd}
                &&B\arrow[dd, "\pi"']\arrow[rrdddd, "\pi'"]\\\\
                A\arrow[rruu, "k"]\arrow[rr, "\epsilon_{A,H}"]\arrow[rrrrdd,"\epsilon_{A,H'}"'] && H\arrow[dashed, rrdd,"\gamma"']\\\\
                &&&&H'
            \end{tikzcd}
        \end{equation}
        \item $H$ has a Hopf algebra structure, 
        \item $\pi$ is a Hopf map,
        \item $H$ is also universal among Hopf algebras in the sense that, for any Hopf algebra $H'$ and Hopf map $\pi':B\to H'$ for which Diagram \eqref{coker-diagram} commutes, there is a unique Hopf algebra homomorphism $\gamma:H\to H'$ such that Diagram \eqref{universal-diagram} commutes.
    \end{enumerate}
    \end{theorem}
    \begin{proof}
        (a) Let $T^\bullet(B)$ denote the tensor algebra of the vector space $B$, and set $H = T^{\bullet}(B)/I$, where $I$ is the two-sided ideal generated by
        \begin{align}
            I = \langle &[b_1S_B(1_B^{(1)})][1_B^{(2)}b_2] - [b_1b_2],\label{HwrtBt}\\
            &[b_11_B^{(1)}][S_B(1_B^{(2)})b_2] - [b_1b_2],\label{HwrtBs}\\
            &[k(a)] - \epsilon_A(a)1_{T^\bullet(B)}\mid b_1, b_2\in B, a\in A\rangle,\label{Hwrtf}
        \end{align}
        where, for $b\in B$, we denote the image of the inclusion of $b$ into $T^{\bullet}(B)$ by $[b]$. Of course, $\pi:B\to H$ is birelatively multiplicative by Relations \eqref{HwrtBt} and \eqref{HwrtBs}. Moreover, Diagram \eqref{coker-diagram}  clearly commutes by Relation \eqref{Hwrtf}.

        (b): Since $\im \pi$ generates $H$, it suffices to specify $\gamma$ on $\im\pi$. Let $\gamma(\pi(b)) = \pi'(b)$. It is clear that Relations \eqref{HwrtBt}-\eqref{Hwrtf} are preserved because $\pi'$ is birelatively multiplicative and satisfies Diagram \eqref{coker-diagram}. Therefore, $\gamma$ is a well-defined unital algebra homomorphism $\gamma:H\to H'$. Moreover, it is clearly unique because again $\im\pi$ generates $H$. 

        (c): We need to give $H$ a compatible comultiplication $\Delta_H:H\to H\otimes H$, counit $\epsilon_H:H\to \bk$, and antipode $S_H:H\to H^{\op}$. Observe that $\Delta$ is multiplicative and $\pi$ is birelatively multiplicative, so, by Proposition \ref{rel-comp-mult},  $(\pi\otimes\pi)\circ\Delta$ is birelatively multiplicative. Moreover, $\epsilon_B$ is birelatively multiplicative. We also have that
        \begin{align*}
            \pi(S_B(1_B^{(2)}b_2))\pi(S_B(b_1S_B(1_B^{(1)}))) &= \pi(S_B(b_2)1_B^{(1)})\pi(S_B(1_B^{(2)})S_B(b_1)) = \pi(S_B(b_1b_2)).
        \end{align*}
        so $\pi\circ S_B:B\to H^{\op}$ is relatively multiplicative with respect to $B_t$. It is easy to verify that it is also relatively multiplicative with respect to $B_s$. Finally, $k$ is a weak Hopf algebra homomorphism, so it is easy to verify that 
        \begin{align*}
            (\pi\otimes\pi)\circ\Delta_B\circ k = \epsilon_{A, H\otimes H},&&\epsilon_B\circ k = \epsilon_{A,\bk},&&\pi\circ S_B\circ k = \epsilon_{A,H^{\op}}.
        \end{align*}
        By (b), this implies there are unique unital algebra homomorphisms $\Delta_H:H\to H\otimes H$, $\epsilon_H:H\to\bk$, and $S_H:H\to H^{\op}$ such that  
        \begin{align}\label{pi-preserves}
            \Delta_H\circ\pi &= (\pi\otimes\pi)\circ\Delta_B, & \epsilon_H\circ\pi &= \epsilon_B, & S_H\circ\pi &= \pi\circ S_B.
        \end{align}
        We leave verifying the coalgebra axioms to the reader, but we show one side of the antipode identity on $\im \pi$ (which consequently implies that it is satisfied on all of $H$):
        $$\pi(b^{(1)})S_H(\pi(b^{(2)}))=\pi(b^{(1)}1_B^{(1)})\pi(S_B(1_B^{(2)})S(b^{(2)})) = \pi(b^{(1)}S_B(b^{(2)})) = \pi(\epsilon_t^B(b)) = \epsilon_H(\pi(b))1_H.$$

        (d): See Equations \eqref{pi-preserves} to see why $\pi$ preserves comultiplication, counit, and antipode.

        (e): A Hopf map is in particular a birelatively multiplicative map, so, by (b), there is a unique unital algebra homomorphism $\gamma:H\to H'$ such that $\gamma\circ\pi = \pi'$. Since both $\pi$ and $\pi'$ preserve $\Delta, \epsilon, S$ and $\im\pi$ is a generating set of $H$, $\gamma$ also preserves these operations. 
    \end{proof}
    
    In addition to being a Hopf map, the cokernel $\pi:B\to H$ satisfies some other useful properties described below.
    \begin{proposition}
        Suppose $k:A\to B$ is a weak Hopf algebra homomorphism. Let $\pi:B\to H$ be its cokernel map. Then, for any $a_1, a_2\in A$ and $b\in B$.
        \begin{itemize}
            \item $\pi(k(a_1)bk(a_2)) = \pi(\epsilon_s(k(a_1))b\epsilon_t(k(a_2)))$,
            \item $\pi(S(k(a_1))bS(k(a_2))) = \pi(\epsilon_t(k(a_1))b\epsilon_s(k(a_2)))$.
        \end{itemize}
    \end{proposition}
        
    Note that the cokernel $H$ is not a quotient of $B$ in general. While $\im\pi$ generates $H$ as an algebra, there are cases where, say, $H$ is infinite-dimensional even when $B$ is not. An example of this is seen in Corollary \ref{groupoid-cokernel}. However, in the cases relevant to us, namely when $\Rep(A)$ is a tensor category, $H$ generally does arise as a quotient of $B$, as coalgebras, by Proposition \ref{simple-surj}. That said, the multiplication on $H$ may be nontrivial to compute directly in general. We will find that to compute the cokernel $H$, typically it is easier to exhibit a well-behaved Hopf map $\pi':B\to H'$ to some candidate cokernel Hopf algebra $H'$ and apply the universal property of $H$. 

    \begin{lemma}\label{projected-image-surj}
        Suppose $k:A\to B$ is a weak Hopf algebra homomorphism and $A$ is connected. Let $\pi:B\to H$ be the cokernel of $k$. Then, for any nonzero $x_1, x_2\in B_t$, $\im\pi = \pi(x_1Bx_2)$.
    \end{lemma}
    \begin{proof}
        By connectivity, there exists $a_1, a_2\in A$ such that $\epsilon_t(k(a_1)x_1) = 1_B = \epsilon_t(k(a_2)x_2)$. The result follows from the following sequence of equalities:
        $$\pi(x_1[S(k(a_1))bk(a_2)]x_2) = \pi(S(k(a_1)x_1)bk(a_2)x_2) = \pi(\epsilon_t(k(a_1)x_1)b\epsilon_t(k(a_2)x_2)) = \pi(b).$$
    \end{proof}
    
    \begin{proposition}\label{simple-surj}
        Suppose $k:A\to B$ is a weak Hopf algebra homomorphism and $A$ is connected. Let $\pi:B\to H$ be the cokernel of $k$. Then, $\im\pi=H$.
    \end{proposition}
    \begin{proof}
        Suppose that $M_n(\bk)$ is an ideal of $B_t$ for some $n\geq 1$. Denote by $e_{ij}\in M_n(\bk)\subseteq B_t$ the elementary matrix with a 1 in the $(i,j)$ slot and 0s elsewhere. The separability idempotent $S(1_B^{(1)})\otimes 1_B^{(2)}$ of $B_t$ restricts to a separability idempotent $e$ of $M_n(\bk)$ and can be verified to be of the form 
        $$e = \sum_{i,j,k=1}^n \lambda_{ij} e_{ki}\otimes e_{jk} = (S(1_B^{(1)})\otimes 1_B^{(2)}I_n),$$
        where $I_n$ is the identity element of $M_n(\bk)$ and $\lambda_{ij}\in\bk$. Because $e^{(1)}e^{(2)}=I_n$, these coefficients also satisfy the property that $\sum \lambda_{ii} = 1_\bk$. Choose $j$ so that $\lambda_{jj}\neq 0$. Then,
        $$S(1_B^{(1)})e_{jj}\otimes e_{jj}1_B^{(2)}e_{jj} = \lambda_{jj} e_{jj}\otimes e_{jj}.$$
        By Lemma \ref{projected-image-surj}, $\pi(e_{jj}B) = \im \pi = \pi(Be_{jj})$. Thus, an arbitrary product of elements in $\im\pi$ may be expressed as
        \begin{align*}
            \pi(b_1e_{jj})\pi(e_{jj}b_2) &= \frac{1}{\lambda_{jj}}\pi(b_1S(1_B^{(1)})e_{jj}) \pi(e_{jj}1_B^{(2)}e_{jj}b_2)\\
            &= \frac{1}{\lambda_{jj}}\pi(b_1S^{-1}(e_{jj})S(1_B^{(1)})) \pi(1_B^{(2)}e_{jj}S(e_{jj})b_2) = \frac{1}{\lambda_{jj}}\pi(b_1S^{-1}(e_{jj})e_{jj}S(e_{jj})b_2).
        \end{align*}
        Therefore, $\im\pi$ is closed under products. It follows that $\im\pi = H$.
    \end{proof} 
    
    \subsection{The kernel of a Hopf map}
    To define kernels of weak Hopf algebra homomorphisms, one would need to dualize the theory presented in Subsection \ref{coker-section}. The kernel map would not be a weak Hopf algebra homomorphism nor a Hopf map. Instead, such a map should have dual properties to that of Hopf maps and induce a tensor functor between the corepresentation categories. We show that kernels for Hopf maps exist and the kernel map, in this case, is a weak Hopf algebra inclusion. We begin with a simple lemma which will allow us to define the kernel.

    \begin{lemma}\label{ker-coalgebra}
        Let $C$ be a coalgebra and $W$ be a vector space. If $\{\pi_j:C\to W\}$ is a collection of linear maps, then 
        $$K = \{c\in C | c^{(1)}\otimes \pi_j(c^{(2)})\otimes c^{(3)} = 0, \forall j\}$$
        is a subcoalgebra of $C$.
    \end{lemma}
    \begin{proof}
        Clearly, $K$ is a vector subspace of $C$ as it is the kernel of a composition of linear maps. For any $c\in K$, write $\Delta(c) = \sum b_i\otimes c_i$, where $b_i\in C$ are linearly independent. Then, by coassociativity, for any $j$, we have
        \begin{align*}
            0 &= c^{(1)}\otimes c^{(2)}\otimes \pi_j(c^{(3)})\otimes c^{(4)}\\
            &= \sum b_i\otimes c_i^{(1)}\otimes \pi_j(c_i^{(2)})\otimes c_i^{(3)}
        \end{align*}
        By linear independence of the $b_i$, $c_i^{(1)}\otimes \pi_j(c_i^{(2)})\otimes c_i^{(3)} = 0$ for all $i,j$, so $c_i\in K$ for all $i$. That is, $\Delta(K)\subseteq C\otimes K$, and, by symmetry, $\Delta(K)\subseteq K\otimes C$, so $\Delta(K)\subseteq (C\otimes K)\cap (K\otimes C)=K\otimes K$.
    \end{proof}
    
    \begin{theorem}[Kernels for Hopf maps]\label{ker}
        Let $\pi:B\to H$ be a Hopf map. Then,
    \begin{enumerate}[(\alph*)]
        \item there is a ``kernel'' weak Hopf algebra $A$ and weak Hopf algebra homomorphism $k:A\to B$ such that the following diagram commutes:
        \begin{equation}\label{ker-diagram}
            \begin{tikzcd}
                B\arrow[ddrr, "\pi"]\\\\
                A\arrow[uu, "k"]\arrow[rr, "\epsilon_{A,H}"'] && H
            \end{tikzcd}
        \end{equation}
        \item $A$ is ``universal'' among such weak Hopf algebras, in the sense that, for any weak Hopf algebra $A'$ and weak Hopf algebra homomorphism $k':A'\to B$ for which Diagram \eqref{ker-diagram} commutes, there is a unique weak Hopf algebra homomorphism $g:A'\to A$ such that the following diagram commutes:
        \begin{equation}\label{ker-universal-diagram}
            \begin{tikzcd}
                &&B\arrow[ddrr, "\pi"]\\\\
                &&A\arrow[uu, "k"']\arrow[rr, "\epsilon_{A,H}"] && H\\\\
                A'\arrow[uuuurr, "k'"]\arrow[uurr, dashed, "g"]\arrow[uurrrr, "\epsilon_{A',H}"']
            \end{tikzcd}
        \end{equation}
    \end{enumerate}
    \end{theorem}
    \begin{proof}
        (a): Let 
        $$A = \{a\in B\mid a^{(1)}\otimes \pi(xa^{(2)}y)\otimes a^{(3)} = a^{(1)}\otimes \epsilon_{B,H}(xa^{(2)}y)\otimes a^{(3)}, \forall x, y\in B_{\min}\},$$ 
        and let $k:A\to B$ be the inclusion. 
        Observe that $A$ is clearly closed under antipode because $\pi$ and $\epsilon_{B,H}$ are antipode-preserving, and the antipode restricts to a linear automorphism of $B_{\min}$. Moreover,
        \begin{align*}
            &(a_1a_2)^{(1)}\otimes \pi(x(a_1a_2)^{(2)}y)\otimes (a_1a_2)^{(3)}\\
            &= [a_1^{(1)}\otimes \pi(xa_1^{(2)}S(1^{(1)}))\otimes a_1^{(3)}][(a_2)^{(1)}\otimes \pi(1^{(2)}(a_2)^{(2)}y)\otimes (a_2)^{(3)}]\\
            &= [a_1^{(1)}\otimes \epsilon_{B,H}(xa_1^{(2)}S(1^{(1)}))\otimes a_1^{(3)}][(a_2)^{(1)}\otimes \epsilon_{B,H}(1^{(2)}(a_2)^{(2)}y)\otimes (a_2)^{(3)}]\\
            &= (a_1a_2)^{(1)}\otimes \epsilon_{B,H}(x(a_1a_2)^{(2)}y)\otimes (a_1a_2)^{(3)}.
        \end{align*}
        Thus, $A$ is closed under multiplication. Of course, $1\in A$ because $\pi|_{B_{\min}} = \epsilon_{A_{\min}, H}$. By Lemma \ref{ker-coalgebra}, $A$ is closed under comultiplication, so $A$ is a weak Hopf subalgebra of $B$.
        
        Finally, Diagram \ref{ker-diagram} commutes:
        $$(\pi\circ k)(a) = (\epsilon\otimes\pi\otimes\epsilon)(a^{(1)}\otimes a^{(2)}\otimes a^{(3)}) = \epsilon_{B,H}(a) = \epsilon_{A,H}(a).$$

        (b): Observe that $k'(a')\in A$ for all $a'\in A'$:
        \begin{align*}
            k'(a')^{(1)}\otimes \pi(xk'(a')^{(2)}y)\otimes k'(a')^{(3)} &= 
            k'((a')^{(1)})\otimes \pi(k'(k'|_{A_{\min}}^{-1}(x)(a')^{(2)}k'|_{A_{\min}}^{-1}(y)))\otimes k'((a')^{(3)})\\
            &= k'((a')^{(1)})\otimes \epsilon(k'(k'|_{A_{\min}}^{-1}(x)(a')^{(2)}k'|_{A_{\min}}^{-1}(y)))\otimes k'((a')^{(3)})\\
            &= k'(a')^{(1)}\otimes \epsilon(xk'(a')^{(2)}y)\otimes k'(a')^{(3)}.
        \end{align*}
        Since $k$ is an inclusion, $g = k'|^{A}:A'\to A$ (the corestriction of $k'$ to $A$) is unique and, since $A$ inherits its weak Hopf algebra structure from $B$, $g$ is a weak Hopf algebra homomorphism.
    \end{proof}
    
    While the cokernel map is not always surjective, the kernel map is always injective.
    
    \begin{proposition}\label{kck=k}
        The kernel of the cokernel of the kernel of a Hopf map $\pi:B\to H$ is isomorphic to the kernel of $\pi$.
    \end{proposition}
    \begin{proof}
        Let $k:A\to B$ be the kernel of $\pi$, let $\pi':B\to H'$ be the cokernel of $k$, and let $k':A'\to B$ be the kernel of $\pi'$. Because $\pi'\circ k = \epsilon_{A, H'}$, by the universal property of $k$, there is a unique weak Hopf algebra homomorphism $g:A'\to A$ such that $k\circ g = k'$. Because $\pi'\circ k = \epsilon_{A, H'}$, by the universal property of $\pi'$, there is a unique Hopf algebra map $\gamma:H\to H'$ such that $\gamma\circ\pi' = \pi$. In particular, $\pi\circ k' = \gamma\circ \pi'\circ k' = \epsilon_{A',H}$. By the universal property of $k'$, there is a unique weak Hopf algebra homomorphism $g':A\to A'$ such that $k'\circ g' = k$. Thus, we have $k\circ g\circ g' = k'\circ g' = k$ and $k'\circ g\circ g' = k\circ g' = k'$  However, $k$ and $k'$ are injective, so $g\circ g' = \id_{A}$ and $g'\circ g = \id_{A'}$.
    \end{proof}
    
    \begin{proposition}
        The cokernel of the kernel of the cokernel of a weak Hopf algebra homomorphism $k:A\to B$ is isomorphic to the cokernel of $k$.
    \end{proposition}
    \begin{proof}
        A similar proof to Proposition \ref{kck=k} applies until we have two Hopf algebra homomorphisms $\gamma:H\to H'$ and $\gamma':H'\to H$ and two Hopf maps $\pi:B\to H$ and $\pi':B\to H'$ such that $\gamma\circ\gamma'\circ\pi' = \pi'$ and $\gamma'\circ \gamma\circ\pi = \pi$. At this point, while $\pi$ and $\pi'$ are not surjective in general, they are right cancellative because $\im\pi$ and $\im\pi'$ generate $H$ and $H'$ respectively.
    \end{proof}

    In both of these propositions, it is actually the case that the isomorphism is just an equality.

    \section{Exact sequences of representation categories}\label{sec:exact-seq}
    \subsection{Normality} 
    Throughout this section, we consider $B$ as a $B$-module with the left regular representation. Let $k:A\to B$ be a weak Hopf algebra homomorphism with cokernel $\pi:B\to H$. The following algebraic condition will turn out to be equivalent to normality of $k^*$ in the case $B$ is finite-dimensional and, in general, imply normality and that $\pi^*$ induces an equivalence between $\Rep(H)$ and $\Ker k^*$,  (assuming $A$ is connected so that normality is defined). We refrain from calling this condition ``normal,'' however, because it may be weaker than normality in the Hopf algebra case.
    \begin{definition}
        We say that a weak Hopf algebra homomorphism $k:A\to B$ is \textit{Inv-normal} if $B\cdot\Inv_A(k^*(B)) = k(A)\cdot\Inv_A(k^*(B))$.
    \end{definition}

    Theorem \ref{full-commute-with-At} gives a condition which implies Inv-normality and generalizes the equality $Bk(A^+)=k(A^+)B$ exhibited by normal Hopf algebra homomorphisms. We apply the following theorem in all of our examples because, under these conditions, the cokernel of a weak Hopf algebra homomorphism is often much easier to identify. It is unclear to the author if Inv-normality is, in fact, equivalent to the condition of Theorem \ref{full-commute-with-At}.
    \begin{theorem}\label{full-commute-with-At}
        Let $k:A\to B$ be a weak Hopf algebra homomorphism. If there is a subspace $V\subseteq B$ such that $B = k(A)V$ and $k(A_t^+)V\subseteq Bk(A_t^+)$, then $k$ is Inv-normal. Moreover, if $B$ is connected, $\epsilon_t^B(v) = \epsilon_B(v)1_B$ for any $v\in V$.
    \end{theorem}
    \begin{proof}
        For $v\in V$ and $a\in A_t^+$, we may write $k(a)v = \sum b_ik(a_i)$ for some $b_i\in B$ and $a_i\in A_t^+$. This implies that, for $\Lambda\in\Inv_A (k^*(B))$, 
        $$k(a)\cdot (v\cdot\Lambda) = \sum b_ik(a_i)\cdot\Lambda = 0,$$
        so $v\cdot\Lambda\in \Inv_{A}(k^*(B))$. It follows that
        $$B\cdot \Inv_A(k^*(B)) = k(A)V\cdot\Inv(k^*(B))\subseteq k(A)\cdot \Inv_A(k^*(B))$$
        Thus, $k$ is Inv-normal.

        For any $v\in V$, the map $B_t\to B_t$ given by $x\mapsto x\epsilon_t(v)$ satisfies
        $$\epsilon_t(k(a)(x\epsilon_t(v))) = \epsilon_t(\epsilon_t(k(a)x)\epsilon_t(v)) + \epsilon_t((k(a)x - \epsilon_t(k(a)x))v) = \epsilon_t(k(a)x)\epsilon_t(v) + 0,$$
        so, if $B$ is connected, $x\mapsto x\epsilon_t(v)$ is a scalar multiple of $\id_{B_t}$. This implies that $\epsilon_t(v) = \epsilon(v)1_B$.
    \end{proof}

    Let $V' = \{v\in B | k(a)v\in Bk(A_t^+), \forall a\in A_t^+\}$. Observe that $V'$ is a unital associative subalgebra of $B$. Moreover, $V'$ contains $k(A_t^+)$, $\Inv_A(k^*(B))$, and any subspace $V$ which satisfies the conditions of Theorem \ref{full-commute-with-At}. In particular, if those conditions are satisfied for some $V$, they are satisfied for $V=V'$. Looking at the $V$ chosen in Theorem \ref{ribbon-exact}, it's clear that $V$ is not unique, as this choice is not an algebra in general. The following corollary is simply a rephrasing of Theorem \ref{full-commute-with-At} in terms of $V'$.
    
    \begin{corollary}
        Let $k:A\to B$ be a weak Hopf algebra homomorphism. If 
            $$B = k(A)\{v\in B | k(a)v\in Bk(A_t^+), \forall a\in A_t^+\},$$ 
        then $k$ is Inv-normal.
    \end{corollary}
    
    In the Hopf algebra case, if there is a $V$ which satisfies the conditions of Theorem \ref{full-commute-with-At}, then $V' = B$:
    $$k(A^+)B = k(A^+)k(A)V = k(A^+)V \subseteq Bk(A^+).$$
    This does not work in the weak Hopf algebra case because $A_t^+$ is not a right ideal of $A$ for a general weak Hopf algebra.
    
    \begin{corollary}
        If $k:A\to B$ is a normal Hopf algebra homomorphism, then $k$ is Inv-normal.
    \end{corollary}
    \begin{proof}
        Let $V = B$ and observe $B = k(1_A)V = k(A)V$. By normality, $k(A_t^+)V=k(A^+)B=Bk(A^+) = Bk(A_t^+)$.
    \end{proof}

    Any finite-dimensional algebra may be written as a subalgebra of $M_m(\bk)$. In the case of the target subalgebra $A_t$, we may embed minimally by $A_t\cong \bigoplus M_{n_i}(\bk)\subseteq M_{m}(\bk)$, where $m=\sum n_i$. The matrix rank is invariant under the choice of inclusion $A_t\incl M_m(\bk)$. Thus, an element $x\in A_t$ has a well-defined rank if we identify $x$ with its image under this inclusion. The following theorem gives us a better handle on the image of the cokernel map in terms of an element of rank 1. This is particularly useful in computing the cokernel in examples, as we will see in Section \ref{sec:examples}. 
    \begin{theorem}\label{V-surj}
        Suppose $k:A\to B$ is a weak Hopf algebra homomorphism and $V\subseteq B$ satisfies the conditions of Theorem \ref{full-commute-with-At}. Suppose $A$ is connected. Let $\pi:B\to H$ be the cokernel map of $k$. Then, there is a rank one element $z\in B_t$ such that $H = \pi(zV)$ and $\epsilon(z) = 1$. 
    \end{theorem}
    \begin{proof}
        Let $M_n(\bk)$ be an ideal of $B_t$. We apply the same notation for elementary matrices as in the proof of Proposition \ref{simple-surj}. By Lemma \ref{projected-image-surj} and \ref{simple-surj}, we know that $H = \pi(B) = \pi(e_{11}B)$. Moreover, for any $b\in B$, $e_{11}\epsilon_t(b) = \sum \lambda_j e_{1j}$, for some $\lambda_j\in\bk$.
        
        One may verify that $S^2$ restricts to a unital algebra automorphism of $M_n(\bk)$. In particular, by the Skolem–Noether theorem, $S^2$ is an inner automorphism of $M_n(\bk)$. Let $x\in M_n(\bk)$ satisfy $S^2(y) = x^{-1}yx$ for any $y\in M_n(\bk)$. Observe that, for any $b\in B$, $a\in A$, and $v\in V$, we also have the following equality
        $$\pi(b\epsilon_t(k(a))v) = \pi(b(k(a) - \epsilon_t(k(a)))v) + \pi(b\epsilon_t(k(a))v) = \pi(bk(a)v).$$
        
        For $b\in B$, we may express $b = \sum_i k(a_i)v_i$, where $a_i\in A$ and $v_i\in V$. Using this decomposition, we have
        \begin{align*}
            \pi(e_{11}b) &= \sum_i \pi(e_{11}k(a_i)v_i) = \sum_i \pi(e_{11}\epsilon_t(k(a_i))v_i).
            \intertext{Suppose $e_{11}\epsilon_t(k(a_i)) = \sum_j \lambda_{ij} e_{1j} $. Then,}
            &= \sum_{i,j} \lambda_{ij}\pi(e_{11}e_{1j}v_i) = \sum_{i,j} \lambda_{ij}\pi(S^{-2}(e_{1j})e_{11}v_i)= \sum_{i,j} \lambda_{ij}\pi(xe_{1j}x^{-1}e_{11}v_i).
            \intertext{Note that $e_{1j}x^{-1}e_{11} = \rho_j e_{11}$ for some $\rho_j\in\bk$, so}
            &= \sum_{i,j} \lambda_{ij}\rho_j\pi(zv_i) = \pi\left(z\sum_{i,j} \lambda_{ij}\rho_jv_i\right),
        \end{align*}
        where $z = xe_{11}$. Thus, $H=\pi(zV)$. Because $x$ is invertible and $e_{11}$ has rank 1, $z$ has rank 1. Finally, observe that, for any $v\in V$, 
        $$\epsilon(zv) = \epsilon(z)\epsilon(v),$$
        so $\epsilon(z)$ must be nonzero, as otherwise $\epsilon(H) = \epsilon(\pi(zV)) = \{0\}$). Consequently, we may normalize so that $\epsilon(z) = 1$.
    \end{proof}
    If $B$ is a weak Hopf algebra and not a Hopf algebra, $z$ is necessarily not unique. We could have chosen $e_{ii}$ rather than $e_{11}$ or a different ideal $M_n(\bk)$ and arrived at a different element $z$. In the face algebra (weak Hopf algebra where $A_t$ is commutative) case, $z$ may be chosen to be any rank one idempotent.
    \begin{proposition}
        Suppose $k:A\to B$ is a face algebra homomorphism and $V\subseteq B$ satisfies the conditions of Theorem \ref{full-commute-with-At}. Suppose $A$ is connected. Let $\pi:B\to H$ be the cokernel map of $k$. Then, for any rank one idempotent $z\in B_t$, we have $H = \pi(zV)$. 
    \end{proposition}
    \begin{proof}
        Let $z\in B_t$ be a rank one idempotent. Then, $\epsilon_t(zb) = \epsilon(zb)z$ for any $b\in B$. Therefore, if $b = \sum k(a_i)v_i\in k(A)V$, we have
        $$\pi(zb) = \sum_i \pi(zk(a_i)v_i) = \sum \pi(z \epsilon_t(k(a_i))v_i) = \sum \epsilon(zk(a_i))\pi(zv_i).$$
        The result follows because $\pi(zB) = H$ by connectivity.
    \end{proof}

    \begin{theorem}\label{normal}
        Suppose $k:A\to B$ is a weak Hopf algebra homomorphism and $A$ is connected. Then, if $k$ is Inv-normal, then $k^*$ is normal. If $k^*$ is normal, then $B\cdot\Inv_A k^*(M) = k(A)\cdot\Inv_A (k^*(M))$ for all $M\in\Rep(B)$. In particular, if $B$ is finite-dimensional, then $k$ is Inv-normal if and only if $k^*$ is normal.
    \end{theorem}
    \begin{proof}
        First, assume $k$ is Inv-normal. By \cite[Lemma 3.3]{ETINGOF20171187}, one need only check the normality condition of $k^*$ on projectives. Every projective $M\in\Rep(B)$ is a submodule of a free module $B^n$. We have the following chain of equalities, which is easy to verify:
        $$k(A)\cdot \Inv_A(k^*(B^n)) =(k(A)\cdot \Inv_A(k^*(B)))^n = (B\cdot \Inv_A(k^*(B)))^n=B\cdot \Inv_A(k^*(B^n)).$$ 
        For any $b\in B$ and $\Lambda\in \Inv_A (k^*(M))$, we have 
        $$b\cdot \Lambda\in M\cap k(A)\cdot \Inv_A (k^*(B^n)) = k(A)\cdot\Inv_A (k^*(M)).$$ 
        Thus, $B\cdot\Inv_A k^*(M) = k(A)\cdot\Inv_A (k^*(M))$ for all projectives $M\in\Rep(B)$. In particular, $k(A)\cdot\Inv_A(k^*(M))$ is closed under the action of $B$. If we consider $N=B\cdot\Inv_A(k^*(M))$ as a $B$-module, then $k^*(N)=k(A)\cdot\Inv_A(k^*(M))$ as an $A$-module, which, by Lemma \ref{Hinv-trivial}, is the maximal trivial $A$-submodule of $k^*(M)$.
 
        Now, suppose $k^*$ is normal. For any $M\in\Rep(B)$, there exists a $B$-submodule $N\subseteq M$, such that $k^*(N) = k(A)\cdot \Inv_A(k^*(M))$, the trivial $A$-submodule of $k^*(M)$. Because $k^*$ fixes $N$ as a set, this says that $N=k(A)\cdot \Inv_A(k^*(M))$ is closed under the action of $B$, so 
        $$k(A)\cdot \Inv_A(k^*(M)) = N = B\cdot N = B\cdot\Inv_A(k^*(M)).$$ 
        In particular, if $B$ is finite-dimensional, this is true for $M=B\in\Rep(B)$, so $k$ is Inv-normal.
    \end{proof}
    If there were a definition of exact sequences of multitensor categories, \cite[Lemma 3.3]{ETINGOF20171187} would likely still hold. In which case, we did not use connectivity anywhere here explicitly. However, our next section heavily depends on the assumption that $A$ is connected.

    \subsection{Kernel agrees with image}
    \begin{theorem}\label{coker-full}
        Suppose $A$ is a connected weak Hopf algebra. Let $k:A\to B$ be a weak Hopf algebra homomorphism and $\pi:B\to H$ be the cokernel map of $k$. Then, $\pi^*:\Rep(H)\incl \Ker k^*$ is a full embedding.
    \end{theorem}
    \begin{proof}
        We already know that $\pi^*$ is faithful by Theorem \ref{tensor-functor}, so we must check that it is also full and that $\pi^*(M)\in \Ker k^*$ for any $M\in\Rep(H)$. 
    
        Let $\rho:\pi^*(M)\to\pi^*(N)$ be a $B$-linear map. Let $\{n_i\}$ be a basis of $N$. Then, $\rho$ is generated by 
        $$\rho:1_B\trel m\mapsto \sum y_i(m)\trel n_i,$$
        where the $y_i:M\to B_t$ are linear maps, only finitely many of which are nonzero for fixed $m$. To verify $A$-linearity, we must show that, for any $a\in A$ and $x\in A_t$,
        $$\epsilon_t(y_i(k(a)\cdot_B m)) = \epsilon_t(k(a)\cdot_B y_i(m)).$$
        Note that, for $a\in A$,
        $$k(a)\cdot_B (1_B\trel m) = \epsilon_t(k(a))\trel m\mapsto \sum_i\epsilon_t(k(a))y_i(m)\trel n_i.$$
        This expression is equal to, by $B$-linearity,
        $$k(a)\cdot_B \sum_i(y_i(m)\trel n_i) = \sum_i\epsilon_t(k(a)y_i(m))\trel n_i.$$
        By comparing tensor factors, we see $\epsilon_t(k(a)y_i(m)) = \epsilon_t(k(a))y_i(m)$ for all $i$. By connectivity, for every $i$ and $m\in M$, there is $c_i(m)\in\bk$ such that $Y_m^i(x) = c_i(m)x$ and $y_i(m) = c_i(m)1_B$. This makes $\rho$ of the form $1_B\trel m\mapsto 1_B\trel n(m)$, where $n(m) = \sum c_i(m)n_i$. 
        
        Next, we show that $n:M\to N$ is $H$-linear. Observe that $\im\pi$ generates $H$, so it is enough to show that $n(\pi(b)\cdot m) = \pi(b)\cdot n(m)$ for all $b\in B$ and $m\in M$. Choose a basis $S$ of $B_s$ containing an element $x'\in S$ with $\epsilon_B(x') = 1_{\bk}$ and no other elements of nonzero counit. Then,
        $$\Delta(1_B) = x'\otimes 1_B + \sum x'_i\otimes x_i,$$
        where each $x_i\in A_t$ and $x'_i\in S\setminus \{x\}$ for all $i$ and $r\in \{s,t\}$. Using this formula for $\Delta(1_B)$, we have
        $$S(1_B^{(1)})\trel \pi(1_B^{(2)}b)m\mapsto S(1_B^{(1)})\trel n(\pi(1_B^{(2)}b)m) = S(x')\trel n(\pi(b)m) + \sum S(x'_i)\trel n(\pi(x_i b)m).$$
        Meanwhile, by $A$-linearity, this expression should be equal to
        $$k\cdot (1_B\trel n(m)) = S(1_B^{(1)})\trel \pi(1_B^{(2)}b)n(m) = S(x')\trel \pi(b)n(m)+ \sum S(x'_i)\trel \pi(x_i b)n(m).$$
        By applying $\epsilon\otimes\id$, we see $n(\pi(b)m) = \pi(b)n(m)$, so $n$ is $H$-linear and $\pi^*(n) = \rho$. Therefore, $\pi^*$ is a full functor. 
        
        Let $M\in\Rep(H)$. Then, by Proposition \ref{rel-comp-mult-functor} and Corollary \ref{epsAB-case}, 
        $$(k^*\circ\pi^*)(M)\cong (\pi\circ k)^*(M) = \epsilon_{A,H}^*(M) \cong  A_t^{\oplus\dim(M)}.$$
        Therefore, $\pi^*:\Rep(H)\incl\Ker k^*$ is a full embedding. 
    \end{proof}
    Once again, we do not use connectedness explicitly in this proof. However, this phenomenon stops in the next lemma, where it is used explicitly several times.
    
    \begin{lemma}\label{r-linear-iso}
        Suppose $k:A\to B$ is weak Hopf algebra homomorphism and $A$ is connected. For $M\in B$-Mod, consider the obvious surjective linear map $r:B_t\otimes \Inv_A(k^*(M))\to B_t\cdot \Inv_A(k^*(M))$. Then, $r$ is an $A$-linear isomorphism, where $A$ acts on $B_t\otimes \Inv_A(k^*(M))$ by $a\cdot (x\otimes \Lambda) = \epsilon_t(k(a)x)\otimes\Lambda$.
    \end{lemma}
    \begin{proof}
        Because we assume that $A$ is connected, trivial $A$-modules are direct sums of copies of $A_t$, so, by Proposition \ref{Hinv-trivial}, $k(A_t)\cdot \Inv_A (k^*(M))$ is isomorphic a direct sum of copies of $A_t$ as an $A$-module. Let $i_\alpha:A_t\to k(A_t)\cdot \Inv_A (k^*(M))$ and $p_\alpha:k(A_t)\cdot \Inv_A (k^*(M))\to A_t$ be compatible inclusions and projections which witness this biproduct. Let $L = \{i_{\alpha}(1_A)\}\subseteq \Inv_A(k^*(M))$. Suppose $\Lambda\in\Inv_A(k^*(M))$. Then, $\Lambda = F(1)$ for some $A$-linear $F:A_t\to k^*(M)$. By connectivity, $f_\alpha = p_\alpha\circ F:A_t\to A_t$ is a scalar multiple of $\id_{A_t}$. In particular, $p_\alpha(\Lambda) = c_{\alpha}1_A$ where $c_\alpha\in\bk$. Thus, $\Lambda = \sum i_\alpha\circ p_{\alpha}(\Lambda) = \sum c_\alpha i_\alpha(1_A)$, so $\Lambda\in\spa (L)$. Of course, the $L$ is linearly independent because $k(A_t)\cdot \Inv_A (k^*(M)) \cong\bigoplus i_\alpha(A_t)$, so $L$ is an $\bk$-basis of $\Inv_A (k^*(M))$. Thus, an arbitrary element of $A_t\otimes \Inv_A k^*(M)$ may be expressed uniquely as $\sum x_i \otimes \Lambda_i$ for $x_i\in A_t$ and $\Lambda_i\in L$. Of course, $L$ is also an $A_t$-basis for $k(A_t)\cdot \Inv_A (k^*(M))$. Therefore, if $\sum x_i \cdot \Lambda_i = 0$ for $\Lambda_i\in L$ and $x_i\in A_t$, then $x_i = 0$ for all $i$. In particular, $\sum x_i\otimes\Lambda_i = 0$, so $r$ is an injection. Finally, observe that 
        $$r(a\cdot x\otimes_B \Lambda) = r(\epsilon_t(k(a)x)\otimes_B\Lambda) = \epsilon_t(k(a)x)\cdot_B\Lambda = a\cdot_A (x\cdot_B\Lambda),$$
        so $r$ is $A$-linear.
    \end{proof}
    
    Let $B'\subseteq B$ be a subalgebra. Observe that if $M, N\in B$-Mod and $f:M\to N$ is a $B$-linear map, then for any $m\in\Inv_A(k^*(M))$, 
    $$k(a)\cdot f(\Lambda) = f(k(a)\cdot \Lambda) = f(\epsilon_t(k(a))\cdot \Lambda) = \epsilon_t(k(a))\cdot f(\Lambda),$$
    so $f(\Lambda)\in \Inv_A(k^*(N))$. Moreover, if $b'\cdot \Lambda \in B'\cdot \Inv_A(k^*(M))$, $f(b'\cdot \Lambda) = b'\cdot f(\Lambda)$, so $f(b'\cdot \Lambda)\in B'\cdot \Inv_A(k^*(N))$. Thus, $M\mapsto B'\cdot \Inv_A(k^*(M))$ defines a linear functor $\Rep(B)\to \Rep(B')$, where morphisms map under restriction. 
    
    Suppose $k:A\to B$ is Inv-normal weak Hopf algebra homomorphism and $A$ is connected. Then, for any $M\in\Rep(B)$, we can define the linear map $\Pi_M:B\cdot \Inv_A(k^*(M))\to \Inv_A(k^*(B))$ by 
    $$\hspace{-5pt}\begin{tikzcd}
    \Pi_M = [B\cdot \Inv_A(k^*(M))=B_t\cdot \Inv_A(k^*(M))\arrow[r, "r^{-1}"]&B_t\otimes \Inv_A(k^*(M))\arrow[r,"\epsilon\otimes \id"]&\Inv_A(k^*(M))].
    \end{tikzcd}$$
    It is not hard to show that $\Pi_M$ is a natural transformation $B\cdot \Inv_A (k^*(-))\To \Inv_A (k^*(-))$, where these are interpreted as functors $\Rep(B)\to\Vec$.
    \begin{lemma}\label{inv-is-birelative}
        Suppose $k:A\to B$ is an Inv-normal weak Hopf algebra homomorphism and $A$ is connected. Then, for any $M\in \Rep(B)$, $H$ acts on $\Inv_A(k^*(M))$ by $\pi(b)\cdot\Lambda = \Pi_M(b\cdot \Lambda)$, and
        \begin{equation}\label{Bt-sep-PiM}
            b\cdot \Lambda = 1_B^{(1)}\cdot \Pi_M(1_B^{(2)}b\cdot \Lambda)
        \end{equation}
        for any $b\in B$ and $\Lambda\in\Inv_A(k^*(M))$.
    \end{lemma}
    \begin{proof}
        We first prove Equation \eqref{Bt-sep-PiM}. Write $b\cdot\Lambda = \sum x_i\cdot\Lambda_i$, where $x_i\in B_t$ and $\Lambda_i\in \Inv_A(k^*(M))$. Then, 
        $$1_B^{(1)}\cdot\Pi_M(1_B^{(2)}b\cdot\Lambda) = \sum_i S(1_B^{(1)})\cdot \Pi_M(1_B^{(2)}x_i\cdot\Lambda_i) = \sum_i\epsilon_t(x_i)\cdot\Lambda_i = \sum_i x_i\cdot \Lambda_i = b\cdot\Lambda.$$
        Next, observe that, because $\epsilon(b_1b_2) = \epsilon(\epsilon_s(b_1)b_2)$, we have
        \begin{align*}
            \Pi_M(b_1b_2\cdot \Lambda) &= \Pi_M(b_11_B^{(1)}\cdot ((\epsilon\otimes\id)\circ r^{-1})(1_B^{(2)}b_2\cdot \Lambda)) \\
            &= \Pi_M(b_11_B^{(1)}\cdot ((\epsilon\otimes\id)((1_B^{(2)}\otimes 1)\cdot r^{-1})(b_2\cdot \Lambda))) \\
            &= \Pi_M(b_11_B^{(1)}\cdot ((\epsilon\otimes\id)((S(1_B^{(2)})\otimes 1)\cdot r^{-1})(b_2\cdot \Lambda))) \\
            &= \Pi_M(b_11_B^{(1)}\cdot \Pi_M(S(1_B^{(2)})b_2\cdot \Lambda)).
        \end{align*}
        Moreover, by applying $\Pi_M(b_1\cdot -)$ to Equation \eqref{Bt-sep-PiM}, we have
        \begin{align*}
            \Pi_M(b_1b_2\cdot \Lambda) &= \Pi_M(b_1S(1_B^{(1)})\cdot \Pi_M(1_B^{(2)}b_2\cdot \Lambda)).
        \end{align*}
        Thus, the map $B\to\End(\Inv_A(k^*(M)))$ given by $b\mapsto \Pi(b\cdot -)$ is birelatively multiplicative. Moreover, $\Pi_M(a\cdot -) = \epsilon_A(a)\id_{\Inv_A(k^*(M))}$. Thus, by Theorem \ref{coker}, $\Pi_M(a\cdot -)$ factors through an algebra homomorphism $H\to \End(\Inv_A(k^*(M)))$. This defines an action of $H$ on $\Inv_A(k^*(M))$ by $\pi(b)\cdot\Lambda = \Pi_M(b\cdot\Lambda)$.
    \end{proof}
    Note that $\Inv_A(k^*(B))$ is an $H$-$B$-bimodule because this space is a right ideal of $B$.
    
    \begin{theorem}\label{pi-induces-equiv}
        Suppose $A$ and $B$ are connected weak Hopf algebras. Let $k:A\to B$ be an Inv-normal weak Hopf algebra homomorphism, and $\pi:B\to H$ be the cokernel map of $k$. Then, $\pi^*:\Rep(H)\overset{\sim}{\to}\Ker k^*$ is an equivalence.
    \end{theorem}
    \begin{proof}
        By Theorem \ref{coker-full}, $\pi^*:\Rep(H)\incl \Ker k^*$ is fully faithful, so it suffices to verify it is essentially surjective. Let $M\in\Ker k^*$, so 
        $M = B_t\cdot \Inv_A(k^*(M)) = B\cdot \Inv_A(k^*(M)).$
        By Lemma \ref{inv-is-birelative}, $\Inv_A(k^*(M))$ has a compatible $H$-module structure. Consider the map $\pi^*(\Inv_A(k^*(M)))\to M$ generated by $x\boxtimes \Lambda\mapsto x\cdot \Lambda$. Note that this is a linear isomorphism for the same reason that $r$ (of Lemma \ref{r-linear-iso}) is a linear isomorphism. This map is also $B$-linear because, by Lemma \ref{inv-is-birelative},
        \begin{align*}
            S(1_B^{(1)})\cdot_B (\pi(1_B^{(2)}bx)\cdot_H \Lambda) = 1_B^{(1)}\cdot_B \Pi_M(1_B^{(2)}bx\cdot_B \Lambda) = bx\cdot \Lambda = b\cdot_B (x\cdot_B \Lambda).
        \end{align*}
        Thus, $\pi^*(\Inv_A(k^*(M)))\cong M$.
    \end{proof}  

    In this section, we only needed that $B\cdot \Inv_A (k^*(M)) = k(A)\cdot \Inv_A (k^*(M))$ for $M\in\Rep(B)$. By Theorem \ref{normal},  this is possibly weaker than Inv-normality, if $B$ is infinite-dimensional.

    \subsection{Dominance}
    In the following lemma, we make the identification $M\otimes_A A^n = M^n$ for right $A$-modules $M$. 
    \begin{lemma}[\cite{pjm/1102983324}]
        Let $A$ be a ring, and let $k:M\to M'$ be an inclusion of right $A$-modules. The following conditions are equivalent:
        \begin{enumerate}[(\alph*)]
            \item for any positive integers $n_1, n_2\geq 1$, right $A$-linear map $T:A^{n_1}\to A^{n_2}$, and element $u'\in (M')^{n_1}$ such that $(k\otimes T)v\in M^{n_2}$, there is $u\in M^{n_1}$ such that $(k\otimes T)u = (k\otimes T)u';$
            \item $\coker(k)$ is flat as a right $A$-module (where this is the cokernel in the category of right $A$-modules);
            \item for any $N\in \Rep(A)$, the $A$-linear map $k\otimes\id_N:M\otimes_A N\to M'\otimes_A N$ is an inclusion;
            \item for any $N\in \Rep(A)$, the sequence $0\to M\otimes_A N\to M'\otimes_A N\to \coker(k)\otimes_A N\to 0$ is exact.
        \end{enumerate}
    \end{lemma}
    When an inclusion $k:M\to M'$ satisfies these equivalent conditions, $k$ is known as \textit{pure}. For a discussion on purity of submodules, including examples and more equivalent conditions due to Cohn, Fieldhouse, Warfield, et. al., see \cite[Section 4J]{lam1999lectures}.

    \begin{proposition}
        Suppose $k:A\to B$ is an algebra inclusion. If $B$ is right free over $A$ or $A$ is von Neumann regular (i.e., for any element $a\in A$, there is $x\in A$ such that $a = axa$), then $k$ is pure.
    \end{proposition}
    Every semisimple algebra is von Neumann regular. Thus, inclusions of finite-dimensional semisimple algebras automatically satisfy the following proposition.
    
    \begin{proposition}
        Let $k:A\to B$ be a pure algebra inclusion. If $B$ is finitely-generated over $A$, then $k^*:\Rep(B)\to\Rep(A)$ is dominant.
    \end{proposition}
    \begin{proof}
        Let $M\in\Rep(A)$. Because $B$ is finitely-generated over $A$, $B\otimes_A M\in\Rep(B)$. Thus, $M\cong A\otimes_A M\to B\otimes_A M=k^*(B\otimes_A M)$ is a well-defined inclusion in $\Rep(A)$. 
    \end{proof}
    This finite-generation property is only necessary so that $B\otimes_A M$ is finite-dimensional.
    
    Suppose $k^*$ is dominant. For any $M\in\Rep(A)$, there is $U\in \Rep(B)$ and an inclusion $\iota:M\incl k^*(U)$ as $A$-modules. We may assume $U = B\iota(M)$. In this case, there is a surjective $B$-linear map $B\otimes_A M\to U$ generated by $b\otimes m\mapsto b\iota(m)$, so $U$ is a quotient of $B\otimes_A M$. Thus, assuming $B$ is finitely-generated over $A$, $k^*$ is dominant if and only if, for every $M\in\Rep(B)$, there is a $B$-linear surjective map $p:B\otimes_A M\to U$ such that $p\circ(k\otimes\id)$ is injective. The author has yet to find a satisfying algebraic characterization of this property, so we focus on the pure case, as this is a well-known condition. 
    
    Combining together the results of this section yields the main theorem of this article.
    \begin{theorem}\label{exact-seq-of-rep-cats}
        Consider a sequence 
        $$\begin{tikzcd}
            A\arrow[r,"k",hook]&B\arrow[r,"\pi", two heads]&H,
        \end{tikzcd}$$
        where 
        \begin{enumerate}[(\alph*)]
            \item $A$ and $B$ are connected weak Hopf algebras and $H$ is a Hopf algebra,
            \item $k:A\incl B$ is an Inv-normal weak Hopf algebra inclusion,
            \item $\pi:B\to H$ is the cokernel map of $k$, and
            \item $B$ is finitely-generated over $A$ and $k$ is pure.
        \end{enumerate}
        Then, the following sequence of tensor categories is exact:
        $$\begin{tikzcd}
            \Rep(H)\arrow[r,hook, "\pi^*"]&\Rep(B)\arrow[r,"k^*"]&\Rep(A).
        \end{tikzcd}$$
    \end{theorem}

    In the case where $A$ and $B$ are Hopf algebras, this theorem simplifies dramatically. Condition (a) is trivially satisfied as Hopf algebras are always connected. Conditions (b) and (c) are implied by the definition of a strictly exact sequence of Hopf algebras. Condition (d) is true in the finite-dimensional case by the Nichols-Z\"oeller freeness theorem. In particular, any strictly exact sequence of finite-dimensional Hopf algebras satisfies the above conditions, so our result is a direct generalization of \cite[Prop. 2.9]{BruguieresNatale11}.
    
    \section{Examples}\label{sec:examples}
        Suppose $k:A\to B$ is a weak Hopf algebra homomorphism such that $A$ is connected. By Theorems \ref{coker} and \ref{V-surj}, to compute the cokernel $H=\im\pi$ of $k$, we may exhibit a surjective Hopf map $B\to H'$ for some Hopf algebra $H'$, which is injective on $zV$ for some $V\subseteq B$ satisfying the conditions of Theorem \ref{full-commute-with-At} and $z$ as in Theorem \ref{V-surj}. In this case, there is a Hopf algebra isomorphism $\pi(zV)\to H'$. We apply this technique in our second and third examples. However, we do the groupoid case explicitly so that we need not assume connectivity.
    
    \subsection{Groupoids}
        Given a groupoid $G$ with finitely many objects, there is a canonical weak Hopf algebra structure on the groupoid algebra $\bk[G]$. Much like group algebras are the canonical example of a Hopf algebra, this is the canonical example of a weak Hopf algebra. Groupoid algebras motivate much of the notation and terminology in the theory of weak Hopf algebras. For example, a finite groupoid is connected if and only if its groupoid algebra is connected.
        
        We denote a morphism $g:s(g)\to t(g)$ in $G$ by $g\in G$ or $g\in\Hom_{G}(s(g)\to t(g))$, where $s(g)$ is the source object of $g$, and $t(g)$ is the target object of $g$. We denote the set of objects of $G$ by $\ob(G)$. Then, the weak Hopf algebra structure on $\ob(G)$ is given as follows: for $g, g_1, g_2\in G$,
        \begin{align*}
            g_1\cdot_{\bk[G]} g_2 &= \begin{cases}
                g_1\cdot_G g_2 & \text{if $s(g_1) = t(g_2)$},\\
                0 & \text{if $s(g_1)\neq t(g_2)$},
            \end{cases}\\
            1_{\bk[G]} &= \sum_{x\in\ob(G)} \id_x,\\
            \Delta(g) &= g\otimes g,\\
            \epsilon(g) &= 1,\\
            S(g) &= g^{-1}.
        \end{align*}
        In this algebra, $1_G = \sum_{X\in\ob(G)} \id_X$. This is precisely why we require finitely many objects. Otherwise, this element would not exist in $\bk[G]$. Of course, the source and target maps send $g\in G$ to the identity on its source and target respectively:
        $$\epsilon_s(g) = \id_{s(g)},\hspace{50pt} \epsilon_t(g) = \id_{t(g)}.$$
        Hence, the target and source subalgebras agree and are generated by the identity morphisms:
        $$\bk[G]_s = \bk[G]_t = \spa\{\id_x\mid x\in\mathrm{ob}(G)\}.$$

        \begin{definition}
            Let $G\subseteq G'$ be a subgroupoid. Then, $G\subseteq G'$ is 
            \begin{itemize}
                \item \textit{wide} if $\ob(G)=\ob(G')$;
                \item \textit{subconnected} if it is wide and the connected components of $G$ and $G'$ agree (by a \textit{connected component} of $G$, we mean an isomorphism class in $G$);
                \item \textit{normal} if it is wide and, for any $g\in G$ and any $g'\in G'$ for which $g'g$ is well-defined, there is $\hat g\in G$ such that $g'g = \hat gg'$.
            \end{itemize}
        \end{definition}
        The notions of wideness and normality are further discussed in \cite{avila2019normalsubgroupoids}.

        \begin{lemma}
            If $f:G\to G'$ is a groupoid homomorphism, then the map $\bk[f]:\bk[G]\to\bk[G']$ is a weak Hopf algebra homomorphism if and only if $f(G)$ is a wide subgroupoid of $G'$. A wide inclusion of groupoids $G\incl G'$ with finitely many objects induces a weak Hopf algebra inclusion $\iota:\bk[G]\incl\bk[G']$. 
        \end{lemma}
        \begin{theorem}
            If $f:G\to G'$ is a homomorphism of groupoids and $f(G)$ is a wide subgroupoid of $G'$. Let $R$ be a set of representatives of the connected components of $f(G)$. Then, the cokernel Hopf algebra $H$ of $\bk[f]:\bk[G]\to\bk[G']$ is the group Hopf algebra of the following group:
            $$\frac{\displaystyle\bigsqcup_{X,Y\in R} \mathrm{FreeGrp}(\Hom_{G'}(X, Y))}{\displaystyle\ncl\left(\left\langle \bigcup_{X,Y\in R} \Hom_{f(G)}(X, Y), g'_1g'_2(g'_1\cdot_{G'} g'_2)^{-1}\text{ composable } g'_1, g'_2\in G'\right\rangle\right)},$$
            where $\bigsqcup$ denotes free product of groups and $\ncl$ denotes the group normal closure inside this free product.
        \end{theorem}
        \begin{proof}
            We compute the cokernel directly as our results on the structure of $H$ do not apply in this general situation. The cokernel $H$ is the quotient of $T^{\bullet}(\bk[G'])$ with the following identifications:
            \begin{align}
                [g_1'g_2'] &= \sum_{X\in\ob(G')} [g_1'S(\id_{X})][\id_{X}g_2'] = [g_1'\id_{s(g_1')}][\id_{s(g_1')}g_2] = \begin{cases}
                    [g_1'][g_2'] & s(g_1') = t(g_2')\\
                    0 & s(g_1')\neq t(g_2').
                \end{cases},\label{Sid-groupoid}\\
                [g_1'g_2'] &= \sum_{X\in\ob(G')} [g_1'\id_{X}][S(\id_{X})g_2'] = [g_1'\id_{s(g_1')}][\id_{s(g_1')}g_2'] = \begin{cases}
                    [g_1'][g_2'] & s(g_1') = t(g_2')\\
                    0 & s(g_1')\neq t(g_2').
                \end{cases},\nonumber\\
                [f(g)] &= \epsilon_{\bk[G]}(g)1_H = 1_H.\label{eps-groupoid}
            \end{align}
            for any $g_1',g_2'\in G'$ and $g\in G$. Relation \eqref{Sid-groupoid} implies that $[g'_1][g'_2][(g'_1\cdot_{G'} g'_2)^{-1}] = 1$ in this quotient for composable $g'_1, g'_2\in G'$. Relation \eqref{eps-groupoid} implies that $[f(G)] = \{1\}$ in the quotient as well.
            
            Relation \eqref{Sid-groupoid} implies that $[g'][(g')^{-1}] = 1$, and Relation \eqref{eps-groupoid} implies that $[\id_X] = 1$ for any $X\in\ob(G')$. Moreover, if there are morphisms $f(g_1)\in\Hom_{f(G)}(X'\to X), f(g_2)\in\Hom_{f(G)}(Y'\to Y),$ and $g'\in \Hom_{G'}(X'\to Y')$, then $[f(g_1)g'f(g_2^{-1})] = [g']$. Thus, any element in $\Hom_{G'}(X'\to Y')$ is identified with an element in $\Hom_{G'}(X\to Y)$, where $X$ and $Y$ are representatives of the isomorphism classes of $X'$ and $Y'$ respectively.  It follows that the map
            $$\bk\left[\frac{\displaystyle\bigsqcup_{X,Y\in R} \mathrm{FreeGrp}(\Hom_{G'}(X, Y))}{\displaystyle\ncl\left(\left\langle \bigcup_{X,Y\in R} \Hom_{f(G)}(X, Y), g'_1g'_2(g'_1\cdot_{G'} g'_2)^{-1}\text{ composable } g'_1, g'_2\in G'\right\rangle\right)}\right]\to H$$
            generated by $g'\mapsto [g']$ is a well-defined surjective algebra homomorphism. 
            
            If $f(g)\in \Hom_{f(G)}(X'\to Y')$, let $X, Y\in R$ be the representatives so that $X\cong X'$, and $Y\cong X'$. We define an inverse map
            $$H\to \bk\left[\frac{\displaystyle\bigsqcup_{X,Y\in R} \mathrm{FreeGrp}(\Hom_{G'}(X, Y))}{\displaystyle\ncl\left(\left\langle \bigcup_{X,Y\in R} \Hom_{f(G)}(X, Y), g'_1g'_2(g'_1\cdot_{G'} g'_2)^{-1}\text{ composable } g'_1, g'_2\in G'\right\rangle\right)}\right]$$
            generated, for $g'\in\Hom_{G'}(X'\to Y')$ by $[g']\mapsto f(g_1)g'f(g_2)^{-1}$, where $f(g_1)\in \Hom_{f(G)}(X'\to X)$ and $f(g_2)\in \Hom_{f(G)}(Y'\to Y)$ are chosen arbitrarily. Note that this choice does not change the image under this map because, if $f(\hat g_1)\in \Hom_{f(G)}(X'\to X)$ and $f(\hat g_2)\in \Hom_{f(G)}(Y'\to Y)$ are different choices, we have
            $$f(g_1)g'f(g_2)^{-1} = f(g_1)f(\hat g_1)^{-1}f(\hat g_1)g'f(\hat g_2)^{-1}f(\hat g_2)f(g_2)^{-1} = f(\hat g_1)g'f(\hat g_2)^{-1},$$
            where the last equality follows from $f(g_1)f(\hat g_1)^{-1}\in \Hom_{f(G)}(X\to X)$ and $f(\hat g_2)f(g_2)^{-1}\in \Hom_{f(G)}(Y\to Y)$, so they are the identity in the quotient. It is easy to verify that this map preserves the relations of $H$, so generates a well-defined algebra homomorphism. By the above argument, this map is invertible. It is obvious that this preserves comultiplication and antipode, so this is a Hopf algebra isomorphism.
        \end{proof}

        \begin{corollary}\label{groupoid-cokernel}
            If $f:G\to G'$ is a homomorphism of groupoids and $f(G)$ is a subconnected subgroupoid of $G'$. Let $R$ be a set of representatives of the connected components of $G'$. The cokernel Hopf algebra $H$ of $\bk[f]:\bk[G]\to\bk[G']$ is 
            $$H\cong \bk\left[\bigsqcup_{X\in R} \Aut_{G'}(X) / \ncl_{\Aut_{G'}(X)}(\Aut_{f(G)}(X))\right],$$
            where $\bigsqcup$ denotes free product and $\ncl_{\Aut_{G'}(X)}(\Aut_{f(G)}(X))$ denotes the normal closure of $\Aut_{f(G)}(X)\leq \Aut_{G'}(X)$.
        \end{corollary}
        A particular case of Corollary \ref{groupoid-cokernel} is, if $f(G)$ is a normal connected subgroupoid of $G$, then
        $$H\cong \bk[\Aut_{G'}(X) / \Aut_{f(G)}(X)]$$
        for any $X\in\ob(G')$. If $G$ is a groupoid, then the map $\bk[G]$ is connected if and only if $G$ is connected, so this case is of the most importance to us. Our notion of subconnected is only used in Corollary \ref{groupoid-cokernel}.

        \begin{remark}
        Given a groupoid $G'$ with finitely many objects and a Hopf algebra $H$, we can characterize Hopf maps $\pi:\bk[G']\to H$ for which $\im\pi$ generates $H$ as well as their kernels. Note that such a map would necessarily have $\Delta(\pi(g')) = \pi(g')\otimes \pi(g')$. Since $\bk[G']$ is generated by $g'\in G'$, $H$ is generated by the grouplikes $\pi(g')$ and hence is a group algebra $H = \bk[Q]$. For $g_1',g_2'\in G'$ composable, 
        $$\pi(g_1'g_2') = \sum_{x\in\ob(G')}\pi(g_1'\id_x)\pi(\id_xg_2') = \pi(g_1'\id_{s(g_1')})\pi(\id_{s(g_1')}g_2') = \pi(g_1')\pi(g_2').$$
        Moreover, $\pi(\id_x)=\epsilon(\id_x)e_Q = e_Q$. In particular, a surjective Hopf map $\pi$ is the linear extension of a surjective groupoid homomorphism $G'\to Q.$ Moreover, any groupoid homomorphism to a group gives rise to a Hopf map.
        
        By Theorem \ref{ker}, the kernel of $\pi$ is given by 
        $$\{a\in \bk[G']\mid a^{(1)}\otimes \pi(xa^{(2)}y)\otimes a^{(3)} = a^{(1)}\otimes \epsilon_{\bk[G'],H}(xa^{(2)}y)\otimes a^{(3)},\forall x,y\in \bk[G']_{\min}\}.$$
        Assuming $a = \sum c_{g'} g'$ for some $g'\in G'$ and $c_{g'}\in \bk$, we have that $a$ is in the kernel of $\pi$ if and only if, for all $X, Y\in\ob(G')$,
        $$\sum c_g g\otimes \pi(\id_X g\id_Y)\otimes g = \sum c_g g\otimes \epsilon_{\bk[G'],H}(\id_X g\id_Y)\otimes g.$$
        If $X\neq t(g)$ or $Y\neq s(g)$, then both sides of the above expression are zero. Otherwise, $\id_X g\id_Y = g$. Thus, the above requirement simplifies to
        $$\sum c_g g\otimes \pi(g)\otimes g = \sum c_g g\otimes 1\otimes g$$
        Thus, we must have that $\pi(g) = 1_H = e_Q$. Moreover, any $g\in G'$ which satisfies this condition is in the kernel. In particular, the kernel of $\pi$ is given by $\spa\{g\in G\mid \pi(g)=e_Q\}.$ The kernel of $\pi$ is therefore a groupoid algebra $\bk[G]$ because it has a basis $G$ which is a subset of $G'$ and is closed under multiplication.
        \end{remark}

        We arrive at the desired result of this section, a condition on a subgroupoid for which the inclusion induces an exact sequence of tensor categories. Note that, for a connected subgroupoid $G$, $\Rep(G)\cong \Rep(\Aut_G(X))$ for any $x\in\ob(G)$. One may show that the following result does not really introduce any new exact sequences of tensor categories beyond those generated by groups.
        \begin{theorem}
            Let $G'$ be a connected finite groupoid and let $G\unlhd G'$ be a normal connected subgroupoid. Then, for any $X\in\ob(G')$ the sequence of tensor categories
            $$\Rep(\Aut_{G'}(X) / \Aut_{G}(X))\to \Rep(G')\to \Rep(G)$$
            is exact.
        \end{theorem}
        \begin{proof}
            First, we prove Inv-normality. Let 
            $$V = \spa\{\sum_{X\in\ob(G')} h_Xg'h_X^{-1}\mid g'\in \End_{G'}(Y), h_X\in\Hom_G(Y\to X), X, Y\in\ob(G')\}.$$ 
            Observe that any $g'\in G'$ may be expressed as
            $$g' = \id_{t(g')}\left(\id_{t(g')}g'\id_{t(g')}^{-1} + \sum_{X\neq t(g')} h_Xg'h_X^{-1}\right),$$
            regardless of our choices of $h_X$, so $\bk[G'] = \bk[G]V$. Moreover, $\bk[G]_t^+ = \spa\{g - \id_{t(g)} | g\in G\}$. Let $g\in G$ and $\sum h_Xg'h_X^{-1}\in V$. Observe that $h_{t(g)}^{-1}gh_{s(g)}\in \End_{G}(t(g'))$, so, by normality, there is $h\in\End_G(t(g'))$ such that $h_{t(g)}^{-1}gh_{s(g)}g' = g'h$. In particular,
            $$(g - \id_{t(g)})\sum_{X\in\ob(G')} h_Xg'h_X^{-1} = h_{t(g)}h_{t(g)}^{-1}gh_{s(g)}g'h_{s(g)}^{-1} - h_{t(g)}g'h_{t(g)}^{-1} = h_{t(g)}g'(hh_{s(g)}^{-1} - h_{t(g)}^{-1}).$$
            This last term is a product of an element $h_{t(g)}g'\in \bk[G']$ and an element $hh_{s(g)}^{-1} - h_{t(g)}^{-1}\in\bk[G]_t^+$, so, by Theorem \ref{full-commute-with-At}, the inclusion $\iota$ is Inv-normal.
            
            Note that, if $G\leq G'$, then $\bk[G']$ is a free (right) $\bk[G]$-module, so the inclusion is in particular pure. By Corollary \ref{groupoid-cokernel}, $H\cong \bk[\Aut_{G'}(X) / \Aut_{G}(X)]$, so the result follows.
        \end{proof}

    \subsection{Formal ribbon extension}
    The notion of a ribbon Hopf algebra was introduced in \cite{ReshetikhinTuraev1990}. This article showed that every quasitriangular Hopf algebra $A$ may be embedded into a canonical ribbon Hopf algebra $\tilde A$, which we dub the \textit{formal ribbon extension of $A$}. This terminology is motivated by the fact that $\Rep(\tilde A)$ is a braided extension of $\Rep(A)$, which is more or less due to Sommerh\"auser \cite{Andruskiewitsch2014hopftensor}. The formal ribbon extension is also defined for weak Hopf algebras \cite{NIKSHYCH200391}. In this section, we show that $\Rep(\tilde A)$ is a braided extension of $\Rep(A)$ in the weak Hopf algebra case using our theory. Along the way, we characterize when $\pi^*$ is braided for a Hopf map $\pi:A\to H$.
    \begin{definition}
        Let $A$ be a weak Hopf algebra. We say that $A$ is \textit{quasitriangular} if there are $R\in \Delta^{\op}(1)(A\otimes A)\Delta(1)$ and $\bar R\in \Delta(1)(A\otimes A)\Delta^{\op}(1)$ such that $R\bar R = \Delta^{op}(1)$ and $R\bar R = \Delta(1)$ and $R$ satisfies the following equations: for any $a\in A$,
        \begin{align*}
            \Delta^{\op}(a)R &= R\Delta(a),\\
            (\id\otimes\Delta)(R) &= R_{13}R_{12}\\
            (\Delta\otimes\id)(R) &= R_{13}R_{23},
        \end{align*}
        where $R_{12} = R\otimes 1$, $R_{13} = R^{(1)}\otimes 1\otimes R^{(2)}$, and $R_{23} = 1\otimes R$.
    \end{definition}
    \begin{lemma}[\cite{NIKSHYCH200391}]
        Let $A$ be a quasitriangular weak Hopf algebra. We have the following identities for the $R$-matrix $R$:
        \begin{align*}
            R\Delta(1) = \Delta^{\op}(1)R &= R, & (\epsilon\otimes\id)(R) = (\id\otimes\epsilon)(R) &= 1, \\
            (\epsilon_s\otimes\id)(R) &= \Delta(1), & 
            (\epsilon_t\otimes\id)(R) &= \Delta^{\op}(1),\\
            (\id\otimes\epsilon_s)(R) &= S(1^{(2)})\otimes 1^{(1)}, & 
            (\id\otimes\epsilon_t)(R) &= S(1^{(1)})\otimes 1^{(2)}.
        \end{align*}
        For any $x\in A_t$, we have
        \begin{align*}
            (1\otimes x)R &= R(x\otimes 1), & (x\otimes 1)R &= (1\otimes S(x))R, & R(1\otimes x) &= R(S(x)\otimes 1).
        \end{align*}
    \end{lemma}

    \begin{proposition}\label{pi-braided}
        Let $A$ be a quasitriangular weak Hopf algebra and $H$ be a quasitriangular Hopf algebra. Suppose $\pi:A\to H$ is a Hopf map. Then, $\pi^*$ is braided if and only if $\pi(R_A^{(1)})\otimes \pi(xR_A^{(2)}) = \epsilon_A(x)R_H$ for all $x\in A_t$.
    \end{proposition}
    \begin{proof}
        Observe that for any Hopf map, we have the following string of equalities:
        \begin{align*}
            &(R_A)_{21}\cdot (1_A^{(2)}\trel n)\otimes (xS(1_A^{(1)})\trel m)\\
            &= (S(1_A^{(1')})\trel \pi(1_A^{(2')}R_A^{(2)}1_A^{(2)})n)\otimes (S(1_A^{(1'')})\trel \pi(1_A^{(2'')}R_A^{(1)}xS(1_A^{(1)}))m)\\
            &= (xS(1_A^{(1')})\trel \pi(1_A^{(2')}R_A^{(2)})n)\otimes (S(1_A^{(1'')})\trel \pi(1_A^{(2'')}R_A^{(1)})m)\\
            &= (xS(1_A^{(1')})\trel \pi(1_A^{(2'')}1_A^{(2')}R_A^{(2)})n)\otimes (S(1_A^{(1'')})\trel \pi(R_A^{(1)})m).
            \intertext{We also have the following:}
            x\trel (R_H)_{21}(n\otimes m)&\mapsto (xS(1_A^{(1)})\trel R_H^{(2)}n)\otimes (1_A^{(2)}\trel R_H^{(1)}m)\\
            &= (xS(1_A^{(1')})\trel \epsilon(1_A^{(2'')}1_A^{(2')})R_H^{(2)}n)\otimes (S(1_A^{(1'')})\trel R_H^{(1)}m).
        \end{align*}
        For $\pi^*$ to preserve the braiding, these two expressions must be equal. In particular, the braiding is preserved if and only if
        $$1_A^{(1)}\otimes 1_A^{(1')}\otimes \pi(1_A^{(2')}1_A^{(2)}R_A^{(2)})\otimes \pi(R_A^{(1)}) = 1_A^{(1)}\otimes 1_A^{(1')}\otimes \epsilon(1_A^{(2')}1_A^{(2)})R_H^{(2)}\otimes R_H^{(1)}.$$
        This is clearly satisfied if the given condition is satisfied. Moreover, applying $\epsilon$ to the second factor shows that, if the braiding is preserved, we must have
        $$1_A^{(1)}\otimes \pi(1_A^{(2)}R_A^{(2)})\otimes \pi(R_A^{(1)}) = 1_A^{(1)}\otimes \epsilon(1_A^{(2)})R_H.$$
        This is only possible if the given property is satisfied.
    \end{proof}
    
    \begin{corollary}
        Let $k:A\to B$ be a weak Hopf algebra homomorphism between two quasitriangular weak Hopf algebras which preserves the $R$-matrix (i.e., $(k\otimes k)(R_A) = R_B$). Let $\pi:B\to H$ denote the cokernel of $k$. Then, $H$ is cocommutative (and hence trivially quasitriangular), and $\pi^*:\Rep(H)\to \Rep(A)$ is braided.
    \end{corollary}
    \begin{proof}
        Observe that $R_B\in (\im k)\otimes(\im k)$, so
        \begin{align*}
            \pi(b^{(1)})\otimes\pi(b^{(2)}) &= \pi(b^{(1)}1_B^{(1)})\otimes\pi(b^{(2)}1_B^{(2)}) = \pi(b^{(1)}R_B^{(2)})\otimes\pi(b^{(2)}R_B^{(1)})\\
            &= \pi(R_B^{(2)}b^{(2)})\otimes\pi(R_B^{(1)}b^{(1)}) = \pi(1_B^{(2)}b^{(2)})\otimes\pi(1_B^{(1)}b^{(1)}) = \pi(b^{(2)})\otimes\pi(b^{(1)}).
        \end{align*}
        Thus, $\im\pi$ is cocommutative, but $\im\pi$ generates $H$, so $H$ is cocommutative. Now, observe that
        $$\pi(xR_B^{(1)})\otimes \pi(R_B^{(2)}) = \pi(x)\otimes 1_H = \epsilon(x)1_H\otimes 1_H.$$
        Thus, $\pi^*$ preserves the braiding by Proposition \ref{pi-braided}.
    \end{proof}

    \begin{definition}
        Suppose $A$ is a quasitriangular weak Hopf algebra with $R$-matrix $R\in \Delta^{\op}(1)(A\otimes A)\Delta(1)$. Then, $A$ is \textit{ribbon} with an invertible \textit{ribbon element} $\mathbf{v}\in Z(A)$ if $S(\mathbf{v}) = \mathbf{v}$ and
        \begin{align*}
            \Delta(\mathbf{v}) &= R_{21}R(\mathbf{v}\otimes\mathbf{v}).
        \end{align*}
    \end{definition}
    
    In a quasitriangular weak Hopf algebra, the \textit{Drinfeld element} $u = S(R^{(2)})R^{(1)}$ is an invertible element which implemeents the square of the antipode: $S^2(a) = uau^{-1}$. Moreover, the element $uS(u)$ is central and $uS(u)^{-1}$ is a group-like which implements the fourth power of the antipode. In a ribbon weak Hopf algebra, $\mathbf{v}^2 = (uS(u))^{-1}$.

    \begin{proposition}
        Let $A$ be a weak Hopf algebra. Then,
        \begin{enumerate}
            \item $R$-matrices on $A$ are in one-to-one correspondence with braidings on $\Rep(A)$,
            \item for a fixed $R$-matrix, ribbon elements on $A$ are in one-to-one correspondence with ribbon structures on $\Rep(A)$.
        \end{enumerate} 
    \end{proposition}
    
    \begin{proposition}[\cite{NIKSHYCH200391}]\label{ribbon-ext}
        Given a quasitriangular weak Hopf algebra $A$, there is a ribbon weak Hopf algebra $\tilde A = A\oplus A\rib$ with the following relations: for any $a\in A$,
        \begin{align*}
            \rib a &= a\rib, & \rib^2 &= (uS(u))^{-1},\\
            \Delta(\rib) &= R_{21}R(\rib\otimes \rib), & \epsilon(\rib a) &= \epsilon(a),\\
            S(\rib) &= \rib.
        \end{align*}
        The inclusion $A\incl\tilde A$ is a weak Hopf algebra inclusion.
    \end{proposition}
    For further discussion of quasitriangular and ribbon weak Hopf algebras, see \cite{NIKSHYCH200391,Zhang2015SovereignAR}. For a discussion of the representation theory of the formal ribbon extension described in Proposition \ref{ribbon-ext} in the case of Hopf algebras, see \cite{kolt2024formalribbonextensionquasitriangular} written by the present article's author. Most of the theory developed may be transported directly to the weak Hopf algebra case. The exception is that we can only now prove the following theorem, which could not be done as directly as in the referenced article. 

    \begin{theorem}\label{ribbon-exact}
        Suppose $A$ is a quasitriangular, connected, weak Hopf algebra. Then, the cokernel Hopf algebra of the inclusion $A\incl\tilde A$ is $\bk[\bZ_2]$, where $g$ is the $\bZ_2$-generator. Moreover, 
        $$\Vec_{\bZ_2}\to \Rep(\tilde A)\to \Rep(A)$$
        is an exact sequence of braided tensor categories, where $\Vec_{\bZ_2}$ has the trivial braiding.
    \end{theorem}
    \begin{proof}
        Observe that $V = \spa\{1, \rib\}$ satisfies the conditions of Theorem \ref{full-commute-with-At}, so the inclusion $A\incl \tilde A$ is Inv-normal. By Theorem \ref{V-surj}, $H$ is generated, as a vector space, by $1$ and $\pi(z\rib)$ for some $z\in \tilde A_t$, so it is at most two-dimensional. There is a Hopf map $\pi':\tilde A\to \bk[\bZ_2]$ given by 
        $$\pi':a_1+a_2\rib\mapsto \epsilon(a_1)e+\epsilon(a_2)g.$$
        It is clear that this preserves antipode, counit, and $\pi'(a) = \epsilon(a)e$ for any $a\in A$. Moreover,
        \begin{align*}
            \pi'((a + a'\rib)^{(1)})\otimes \pi'((a + a'\rib)^{(2)}) &= \epsilon(a^{(1)})\otimes\epsilon(a^{(2)}) + \epsilon((a')^{(1)}R^{(2)}R^{(1)})g\otimes \epsilon((a')^{(2)}R^{(1)}R^{(2')})g\\
            &= \epsilon(a^{(1)})\otimes\epsilon(a^{(2)}) + \epsilon(1^{(2)}(a')^{(2)}1^{(2')})g\otimes \epsilon(1^{(1)}(a')^{(1)}1^{(1')})g\\
            &= \epsilon(a)1\otimes 1 + \epsilon(a')g\otimes g = \Delta(\pi(a + a'\rib)).
        \end{align*}
        Moreover, $\pi'$ is birelatively multiplicative. We verify relative multiplicativity with respect to $\tilde A_t$.
        \begin{align*}
            \pi'((a_1+a_1'\rib)S(1^{(1)}))\pi'(1^{(1)}(a_2+a_2'\rib)) &= \epsilon(a_1a_2 + a_1'a_2') + \epsilon(a_1a_2' + a_1'a_2)g\\
            &= \epsilon(a_1a_2 + a_1'a_2'(uS(u))^{-1}) + \epsilon(a_1a_2' + a_1'a_2)g\\
            &= \pi((a_1+a_1'\rib)(a_2+a_2'\rib)).
        \end{align*} 
        The case of $\tilde A_s$ is identical. Note that $\pi':B\to \bk[\bZ_2]$ is surjective and hence factors through a surjective Hopf algebra homomorphism $\gamma:H\to\bk[\bZ_2]$. However, $H$ is at most 2-dimensional, so $\gamma$ is an isomorphism.
        
        It is clear that $\tilde A$ is finitely generated and free over $A$. Thus, the result follows by Theorem \ref{exact-seq-of-rep-cats}.
    \end{proof}
    \subsection{Crossed products}
    A strictly exact sequence $A\to B\to H$ of Hopf algebras is equivalent to a cocycled crossed product of $A$ and $H$ \cite{BlattnewMontgomery1989}. That is, given a strictly exact sequence, there is a weak action $\cdot$ of $H$ on $A$ and a compatible cocycle $\sigma$ so that $B\cong A\#_{\sigma} H$. Moreover, any cocycled crossed product Hopf algebra $A\#_\sigma H$ fits into a canonical strictly exact sequence $A\to A\#_\sigma H\to H$. 
    
    In this section, we discuss cocycled crossed products when $A$ is a weak Hopf algebra (but $H$ is still a Hopf algebra to align with our theory). We give sufficient conditions under which $A\#_\sigma H$ is a weak Hopf algebra and show that, assuming $H$ is finite-dimensional and $A$ is connected, there is an exact sequence of representation categories
    $$\Rep(H)\to\Rep(A\#_\sigma H)\to \Rep(A).$$
    Note: because the conditions of Theorem \ref{exact-seq-of-rep-cats} do not agree with a strictly exact sequence in the Hopf algebra case, we should expect no correspondence between sequences which satisfy these conditions and cocycled crossed products. 
    
    \begin{definition}\label{weak}
        Let $H$ be a Hopf algebra and $A$ be a unital associative algebra. A \textit{weak action of $H$ on $A$} is a linear map $\cdot: H\otimes A\to A$ such that, for all $a, a_1, a_2\in A$ and $h, h_1, h_2\in H$,
        \begin{enumerate}
            \item $h\cdot (a_1 a_2) = (h^{(1)}\cdot a_1)(h^{(2)}\cdot a_2)$,
            \item $h\cdot 1_A = \epsilon(h)1_A$,
            \item $1_H\cdot a = a$.
        \end{enumerate}
        A weak action is \textit{symmetric} if, for all $a\in A$ and $h\in H$, 
        $$h^{(1)}\otimes h^{(2)}\cdot a = h^{(2)}\otimes h^{(1)}\cdot a.$$ 
    \end{definition}
    \begin{definition}\label{cocycle}
        Suppose $H$ acts weakly on $A$ via $\cdot:H\otimes A\to A$. A linear map $\sigma:H\otimes H\to A$
        \begin{enumerate}
            \item is \textit{normal} if, for all $h\in H$, $\sigma(h\otimes 1_H) = \sigma(1_H\otimes h) = \epsilon(h)1_A$;
            \item is a \textit{cocycle} if, for all $h_1, h_2, h_3\in H$, 
            $$h_1^{(1)}\cdot \sigma(h_2^{(1)}\otimes h_3^{(1)})\sigma(h_1^{(2)}\otimes (h_2^{(2)}h_3^{(2)})) = \sigma(h_1^{(1)}\otimes h_2^{(1)})\sigma((h_1^{(2)}h_2^{(2)})\otimes h_3);$$
            \item satisfies the \textit{twisted module condition} if, for all $h_1, h_2\in H$ and $a\in A$,
            $$h_1^{(1)}\cdot (h_2^{(1)}\cdot a)\sigma(h_1^{(2)}\otimes h_2^{(2)}) = \sigma(h_1^{(1)}\otimes h_2^{(1)})((h_1^{(2)}h_2^{(2)})\cdot a);$$
            \item is \textit{symmetric} if, for all $h_1, h_2\in H$,
            $$h_1^{(1)}h_2^{(1)}\otimes \sigma(h_1^{(2)}\otimes h_2^{(2)})=h_1^{(2)}h_2^{(2)}\otimes \sigma(h_1^{(1)}\otimes h_2^{(1)}).$$
        \end{enumerate}
    \end{definition}
    Note that we do not assume that a weak action $\cdot:H\otimes A\to A$ satisfies $h_1\cdot (h_2\cdot a) = (h_1h_2)\cdot a$, so $A$ is not a module over $A$. This is the source of weakness in Definition \ref{weak}. A map $\sigma:H\otimes H\to A$ which satisfies the twisted module condition allows us to achieve this equality up to a twist by $\sigma$. 

    \begin{lemma}[{\cite[Lem. 4.4--5]{BlattnerCohen1986}}]\label{crossed-algebra}
        Suppose $\cdot: H\otimes A\to A$ is a weak action of a Hopf algebra $H$ on a unital associative algebra $A$. Let $\sigma:H\otimes H\to A$ be a linear map and denote by $A\#_\sigma H$ the vector space $A\otimes H$ with simple tensors denoted $a\# h$. Define a multiplication on $A\#_\sigma H$ as follows: for any $a_1, a_2\in A$ and $h_1, h_2\in H$, 
        $$(a_1\# h_1)(a_2\# h_2) = [a_1 (h_1^{(1)}\cdot a_2)\sigma(h_1^{(2)}\otimes h_2^{(1)})]\# [h_1^{(3)}h_2^{(2)}].$$
        Then, this multiplication defines a unital algebra structure on $A\#_\sigma H$ if and only if $\sigma$ is a normal cocycle with the twisted module condition.
    \end{lemma}

    \begin{theorem}
        Let $\cdot: H\otimes A\to A$ be a weak action of a Hopf algebra $H$ on a weak Hopf algebra $A$ and $\sigma:H\otimes H\to A$ be a normal cocycle with the twisted module condition. Endow $A\#_\sigma H$ with the algebra structure in Lemma \ref{crossed-algebra}. Suppose $\cdot$ and $\sigma$ are both symmetric coalgebra homomorphisms. Suppose further $S_A\circ \sigma$ is the convolution inverse of $\sigma$ (equivalently, $\epsilon_t^A(\sigma(h_1\otimes h_2)) = \epsilon_s^A(\sigma(h_1\otimes h_2)) =\epsilon_H(h_1)\epsilon_H(h_2)1_A$). Then, $A\#_\sigma H$ is a weak Hopf algebra with the tensor product coalgebra structure. The antipode is given by 
        $$S(a\# h) = [S_A(\sigma(S_H(h^{(2)})\otimes h^{(3)}))\# S_H(h^{(1)})][S_A(a)\# 1].$$ 
        Moreover, if $A$ is connected and $H$ is finite-dimensional, there is an exact sequence of tensor categories
        $$\Rep(H)\to \Rep(A\#_\sigma H)\to \Rep(A).$$
    \end{theorem}
    \begin{proof}
        Coassociativity and counitality are obvious by construction. The preservation of multiplication under $\Delta$ works identically to the case where $A$ is a Hopf algebra, which is well-known.

        By comultiplicativity, we see that
        $$h^{(1)}\cdot 1_A^{(1)}\otimes h^{(2)}\cdot 1_A^{(2)} = (h\cdot 1_A)^{(1)}\otimes (h\cdot 1_A)^{(2)} = \epsilon(h)1_A^{(1)}\otimes 1_A^{(2)}.$$ 
        Therefore, by symmetry,
        \begin{align*}
            h\cdot \epsilon_t(a) &= h\cdot 1_A^{(2)}\epsilon(1_A^{(1)}a) = h^{(1)}\cdot 1_A^{(2)}\epsilon((h^{(2)}\cdot 1_A^{(1)})(h^{(3)}\cdot a)) = h^{(2)}\cdot 1_A^{(2)}\epsilon((h^{(1)}\cdot 1_A^{(1)})(h^{(3)}\cdot a))\\
            &= (h^{(1)}\cdot 1_A)^{(2)}\epsilon((h^{(1)}\cdot 1_A)^{(1)}(h^{(2)}\cdot a)) = 1_A^{(2)}\epsilon(1_A^{(1)}(h\cdot a)) = \epsilon_t(h\cdot a).
        \end{align*}
        Next, we show weak multiplicativity of the counit:
        \begin{align*}
            &\epsilon((a_1\# h_1)(a_2^{(1)}\# h_2^{(1)}))\epsilon((a_2^{(2)}\# h_2^{(2)})(a_3\# h_3))\\
            &= \epsilon((a_1(h_1^{(1)}\cdot a_2^{(1)})\sigma(h_1^{(2)}\otimes h_2^{(1)})\# h_1^{(3)}h_2^{(2)}))\epsilon((a_2^{(2)}(h_2^{(3)}\cdot a_3)\sigma(h_2^{(4)}\otimes h_3^{(1)})\# h_2^{(5)}h_3^{(2)}))\\
            &= \epsilon_A(a_1(h_1^{(1)}\cdot a_2^{(1)}))\epsilon_H(h_1^{(2)}h_2^{(1)})\epsilon_A(a_2^{(2)}(h_2^{(2)}\cdot a_3))\epsilon_H(h_2^{(3)}h_3)\\
            &= \epsilon_A(a_1(h_1^{(1)}\cdot a_2^{(1)}))\epsilon_A((h_1^{(2)}\cdot a_2^{(2)})h_1^{(3)}\cdot(h_2\cdot a_3))\epsilon_H(h_3)\\
            &= \epsilon(a_1(h_1^{(1)}\cdot a_2)\sigma(h_1^{(2)}\otimes h_2^{(1)})((h_1^{(3)}h_2^{(2)})\cdot a_3)\sigma(h_1^{(4)}h_2^{(3)}\otimes h_3^{(1)})\# h_1^{(5)}h_2^{(4)}h_3^{(2)}).
        \end{align*}
        Showing weak multiplicativity in the other direction follows by the same argument, applying symmetry of $\cdot$.
        Next, weak comultiplicativity of the unit may be seen as
        \begin{align*}
            (1_A^{(1)}\# 1_H)\otimes (1_A^{(2)}\# 1_H)\otimes (1_A^{(3)}\# 1_H) &= (1_A^{(1)}\# 1_H)\otimes (1_A^{(2)}1_A^{(1')}\# 1_H)\otimes (1_A^{(2')}\# 1_H)\\
            &= [(1_A^{(1)}\# 1_H)\otimes (1_A^{(2)}\# 1_H)] [(1_A^{(1')}\# 1_H)\otimes (1_A^{(2')}\# 1_H)].
        \end{align*}
        Thus, $A\#_\sigma H$ is a weak bialgebra. Now, we verify the antipode identities. First observe that we have the following identity involving the cocycle $\sigma$:
        \begin{align*}
            \epsilon(h)1_A &= \sigma(h^{(1)}\otimes S(h^{(2)})h^{(3)}) \\
            &= S(h^{(1)}\cdot \sigma(S(h^{(4)})^{(1)}\otimes h^{(5)}))h^{(2)}\cdot\sigma(S(h^{(4)})^{(2)}\otimes h^{(6)}) \sigma(h^{(3)}\otimes S(h^{(4)})^{(3)}h^{(7)}) \\
            &= S(h^{(1)}\cdot \sigma(S(h^{(4)})^{(1)}\otimes h^{(5)}))\sigma(h^{(2)}\otimes S(h^{(4)})^{(2)}) \sigma(h^{(3)}S(h^{(4)})^{(3)}\otimes h^{(6)})\\
            &= S(h^{(1)}\cdot \sigma(S(h^{(3)})^{(1)}\otimes h^{(4)}))\sigma(h^{(2)}\otimes S(h^{(3)})^{(2)}).
        \end{align*}
        Therefore, we may verify the target antipode axiom as follows:
        \begin{align*}
            & (a^{(1)}\# h^{(1)})S(a^{(2)}\# h^{(2)}) \\
            &= [a^{(1)}\# h^{(1)}][S(\sigma(S(h^{(3)})\otimes h^{(4)}))\#S(h^{(2)})][S(a^{(2)})\#1_H]\\
            &= [a^{(1)}h^{(1)}\cdot S(\sigma(S(h^{(5)})\otimes h^{(6)}))\sigma(h^{(2)}\otimes S(h^{(4)})^{(1)})\# h^{(3)}S(h^{(4)})^{(2)}][S(a^{(2)})\#1_H]\\
            &= [a^{(1)}h^{(1)}\cdot S(\sigma(S(h^{(6)})\otimes h^{(7)}))\sigma(h^{(2)}\otimes S(h^{(5)}))\# h^{(3)}S(h^{(4)})][S(a^{(2)})\#1_H]\\
            &= [a^{(1)}S(h^{(1)}\cdot \sigma(S(h^{(3)})^{(1)}\otimes h^{(4)}))\sigma(h^{(2)}\otimes S(h^{(3)})^{(2)})\# 1_H][S(a^{(2)})\#1]\\
            &= \epsilon_t^A(a)\# \epsilon_H(h)1_H = \epsilon_t(a\# h). 
        \end{align*}
        Similarly, the source antipode axiom may be verified as follows:
        \begin{align*}
            & S(a^{(1)}\# h^{(1)})(a^{(2)}\# h^{(2)}) \\
            &= [S(\sigma(S(h^{(2)})\otimes h^{(3)}))\#S(h^{(1)})][\epsilon_s(a)\# h^{(4)}]\\
            &= [S(\sigma(S(h^{(2)})\otimes h^{(3)})) (S(h^{(1)})^{(1)}\cdot\epsilon_s(a))\sigma(S(h^{(1)})^{(2)}\otimes h^{(4)})\# S(h^{(1)})^{(3)}h^{(5)}]\\
            &= S(\sigma(S(h^{(4)})\otimes h^{(5)})) (S(h^{(3)})\cdot (h^{(7)}\cdot 1_A^{(1)}\epsilon(ah^{(6)}\cdot 1_A^{(2)}))\sigma(S(h^{(2)})\otimes h^{(8)})\# S(h^{(1)})h^{(9)}\\
            &= \epsilon(ah^{(6)}\cdot 1_A^{(2)}))S(\sigma(S(h^{(4)})\otimes h^{(5)})) \sigma(S(h^{(3)})\otimes h^{(7)})(S(h^{(2)})h^{(8)})\cdot 1_A^{(1)}\# S(h^{(1)})h^{(9)}\\
            &= \epsilon(ah^{(1)}\cdot 1_A^{(2)}))1_A^{(1)}\# S(h^{(2)})h^{(3)} = \epsilon(ah\cdot 1_A^{(2)}))1_A^{(1)}\# 1_H = \epsilon_s(a\# h).
        \end{align*}
        The final antipode identity may be verified using the fact that $A$ satisfies the identity as well. Thus, $A\#_\sigma H$ is a weak Hopf algebra.
        
        Finally, we verify the conditions of Theorem \ref{exact-seq-of-rep-cats}. We first show that $\epsilon\otimes\id:A\#_\sigma H\to H$ is a surjective Hopf map. For any $a_1, a_2\in A, h_1, h_2\in H$, we have
        \begin{align*}
            &(\epsilon\otimes \id)((a_1\# h_1)(S(1_A^{(1)})\# 1_H))(\epsilon\otimes \id)((1_A^{(2)}\# 1_H)(a_2\# h_2))\\
            &= \epsilon(a_1h_1^{(1)}\cdot S(1_A^{(1)}))h_1^{(2)}\epsilon(1_A^{(2)}a_2)h_2 = \epsilon(a_1 h_1^{(1)}\cdot \epsilon_t(a_2))h_1^{(2)}h_2\\
            &= \epsilon(a_1 h_1^{(1)}\cdot a_2)h_1^{(2)}h_2 = \epsilon(a_1 (h_1^{(1)}\cdot a_2)\sigma(h_1^{(2)}\otimes h_2^{(1)}))h_1^{(3)}h_2^{(2)} \\
            &= (\epsilon\otimes \id)((a_1\#h_1)(a_2\#h_2)),\\
            &(\epsilon\otimes \id)((a_1\# h_1)(1_A^{(1)}\# 1_H))(\epsilon\otimes \id)((S(1_A^{(2)})\# 1_H)(a_2\# h_2))\\
            &= \epsilon(a_1h_1^{(1)}\cdot 1_A^{(1)})h_1^{(2)}\epsilon(S(1_A^{(2)})a_2)h_2 = \epsilon(a_1h_1^{(1)}\cdot \epsilon_s(S^{-1}(a_2)))h_1^{(2)}h_2\\
            &= \epsilon(a_1(h_1^{(1)}\cdot a_2))h_1^{(2)}h_2= (\epsilon\otimes \id)((a_1\#h_1)(a_2\#h_2)),\\
            \intertext{Thus, $\epsilon\otimes\id$ is birelatively multiplicative. Moreover, this map is antipode preserving:}
            (\epsilon\otimes\id)(S(a\#h)) &= \epsilon(S_A(\sigma(S_H(h^{(3)})\otimes h^{(4)}))S_A(h^{(1)})\cdot a))h^{(2)} = \epsilon(a)S(h).
        \end{align*}
        Preservation of comultiplication and counit and the fact that $(\epsilon\otimes \id)|_{A\# 1_H} = \epsilon_{A,H}$ are easy to verify. Thus, $\epsilon\otimes\id$ factors through a surjective Hopf algebra homomorphism $\im\pi\to H$.  

        Observe that $V = 1_A\# H$ satisfies the properties of Theorem \ref{full-commute-with-At}; of course, $A\#_\sigma H = (A\# 1_H)(1_A\# H)$ and for $a\in A_t^+$ and $h\in H$,
        \begin{align*}
            (a\# 1)(1\# h) &= (h^{(3)}S(h^{(4)})^{(3)}\cdot a)S(\sigma(h^{(2)}\otimes S(h^{(4)})^{(2)})) S^{2}(h^{(1)}\cdot \sigma(S(h^{(4)})^{(1)}\otimes h^{(5)}))\# h^{(6)}\\
            &= S(\sigma(h^{(3)}\otimes S(h^{(4)})^{(3)}))(h^{(2)}\cdot S(h^{(4)})^{(2)}\cdot a) (h^{(1)}\cdot S^2(\sigma(S(h^{(4)})^{(1)}\otimes h^{(5)})))\# h^{(6)}\\
            &= (S(\sigma(h^{(1)}\otimes S(h^{(2)})^{(3)}))\# h^{(4)})((S(h^{(2)})^{(2)}\cdot a)S^2(\sigma(S(h^{(2)})^{(1)}\otimes h^{(3)}))\# 1_H).
        \end{align*}
        It is not hard to see that the second factor is always zero under $\epsilon_t$ because $\epsilon_t(\sigma(h_1\otimes h_2)) = \epsilon(h_1)\epsilon(h_2)1_A$ and $\epsilon_t(h\cdot a) = h\cdot \epsilon_t(a)$. Thus, $A\incl A\#_\sigma H$ is Inv-normal. By Theorem \ref{V-surj}, $\pi(z\# H) = \im\pi$ for some $z$ with counit 1. If $\epsilon(z)h_1 = \epsilon(z)h_2$, then of course $h_1=h_2$. Thus, $\epsilon\otimes\id$ factors through a Hopf algebra isomorphism $\im\pi\cong H$. Right freeness is \cite[Thm. 1.18]{BlattnewMontgomery1989} noting that $S\circ \sigma$ is a convolution inverse of $\sigma$ and $S$ is invertible because $H$ is finite-dimensional. Finally, $A\#_\sigma H$ is finitely generated over $A$ because again $H$ is finite-dimensional. Thus, the result follows by Theorem \ref{exact-seq-of-rep-cats}.
    \end{proof}
    A more satisfying result would be a set of algebraic properties which are equivalent to $A\#_\sigma H$ having a weak Hopf algebra structure. This construction may be generalizable to a case where $\cdot$ is not counit-preserving and/or $S\circ \sigma$ is not the convolution inverse of $\sigma$. Unfortunately, for example, the condition that $\cdot$ must be counit-preserving was not successfully verified as the technique in the Hopf algebra case assumes that $\Delta(1_A) = 1_A\otimes 1_A$.

    \section*{Acknowledgements}
        The author is supported by the National Science Foundation Graduate Research Fellowship Program under Grant No. 2139319. Any opinions, findings, and conclusions or recommendations expressed in this material are those of the author and do not necessarily reflect the views of the National Science Foundation.

        The author thanks C. Jones and Z. Wang for useful comments. No artificial intelligence was used in the creation of this manuscript.

\bibliographystyle{abbrv}
\bibliography{bib}

\end{document}